\documentclass[11pt]{article} 
\usepackage{setspace}

\usepackage[utf8]{inputenc} 

\usepackage{subfigure}
\usepackage[margin=1in]{geometry}
\usepackage{graphicx} 
\usepackage{epstopdf}

\usepackage{booktabs} 
\usepackage{array} 
\usepackage{paralist} 
\usepackage{verbatim} 
\usepackage{subfig} 

\usepackage{fancyhdr} 
\usepackage[nottoc,notlof,notlot]{tocbibind} 
\usepackage[titles,subfigure]{tocloft} 

\usepackage{amsmath, amsthm, amssymb}

\usepackage{multirow}
\usepackage{algorithm}
\usepackage{algorithmic}

\usepackage{amsthm}
\newtheorem{theorem}{Theorem}

\newtheorem{proposition}{Proposition}

\allowdisplaybreaks

\DeclareMathOperator*{\argmax}{argmax}
\DeclareMathOperator*{\rank}{rank}
\DeclareMathOperator*{\conv}{conv}
\DeclareMathOperator*{\extr}{extr}
\DeclareMathOperator*{\proj}{proj}
\usepackage{rotating}
\usepackage{pdflscape}
\usepackage{pgfplots}
\usepgfplotslibrary{units}
\usepackage{url}
\usepackage{color}

\graphicspath{{figures/}} 

\allowdisplaybreaks

\usepackage{color}
\usepackage{bm}

\newcommand{\cI}{\mathcal{I}}

\newcommand{\cU}{\mathcal{U}}
\newcommand{\cD}{\mathcal{D}}
\newcommand{\cP}{\mathcal{P}}
\newcommand{\cG}{\mathcal{G}}
\newcommand{\cM}{\mathcal{M}}

\title{A study of rank-one sets with linear side constraints and application to the pooling problem
}
\author{Santanu S. Dey\thanks{santanu.dey@isye.gatech.edu, H. Milton  Stewart School of Industrial \& Systems Engineering, Georgia Institute of Technology, Atlanta, GA 30332.} \and 
Burak Kocuk\thanks{burakkocuk@sabanciuniv.edu, Industrial Engineering Program,  Sabanc{\i} University, Istanbul, Turkey 34956.} \and 
Asteroide Santana\thanks{asteroide.santana@gatech.edu, H. Milton  Stewart School of Industrial \& Systems Engineering, Georgia Institute of Technology, Atlanta, GA 30332.}}

\begin{document}

\maketitle

\begin{abstract}
We study sets defined as the intersection of a rank-1 constraint with different choices of linear side constraints. We identify different conditions on the linear side constraints, under which the convex hull of the rank-1 set is polyhedral or second-order cone representable. In all these cases, we also show that a linear objective can be optimized in polynomial time over these sets. Towards the application side, we show how these sets relate to commonly occurring substructures of a general quadratically constrained quadratic program. To further illustrate the benefit of studying quadratically constrained quadratic programs from a rank-1 perspective, we propose new rank-1 formulations for the generalized pooling problem and use our convexification results to obtain several new convex relaxations for the pooling problem. Finally, we run a comprehensive set of computational experiments and show that our convexification results together with discretization significantly help in improving dual bounds for the generalized pooling problem.

\end{abstract}

\section{Introduction}
\subsection{Motivation}
A general quadratically constrained quadratic program (QCQP) is an optimization problem of the following form:
\begin{eqnarray}\label{eq:qcqp}
\begin{array}{rl}
\textup{min} & x^{\top}Q^0x + (a^0)^{\top}x\\
\textup{s.t.} &  x^{\top}Q^k x + (a^k)^{\top}x \leq b_k  \qquad \forall k \in \{1, \dots, m \} \\
& x \in [0, \ 1]^n,
\end{array}
\end{eqnarray}
where the matrices $Q^i$ for $i \in \{ 0, \dots, m\}$ are not assumed to be positive semi-definite.  

Building convex relaxations of the feasible region of a QCQP is a key direction of research. Many general-purpose convexification schemes have been proposed for QCQPs. It includes, for example, the Reformulation-Linearization Technique (RLT)~\cite{sherali2013reformulation}, the Lasserre hierarchy~\cite{lasserre2001global}, and linear programming (LP) and second-order cone programming (SOCP) based alternatives to sum of squares optimization~\cite{ahmadi2014dsos}.  An important area of study regarding convexification schemes for QCQPs is to convexify commonly occurring substructures, like in the case of integer programming. However, most of the work in this direction in the global optimization area has been focused on convexification of functions (i.e., finding convex and concave envelopes), see for example~\cite{al1983jointly,ryoo2001analysis,liberti2003convex,benson2004concave,meyer2004trilinear,belotti2010valid,bao2015global,misener2015dynamically,
del2016polyhedral,sherali1997convex,Rikun1997,meyer2005convex,tawarmalani2002convexification,tuy2016convex,buchheim2017monomial,crama2017class,
Adams2018,gupte2017extended}.
There are relatively lesser number of results on convexification of sets~\cite{tawarmalani2013explicit,nguyen2013deriving,nguyen2011convexification,tawarmalani2010strong,
akshayguptethesis,kocuk2018matrix,rahman2017facets,davarnia2017simultaneous,LiVittal2018,Burer2017,modaresi2017convex,dey2018new}. It is well-known that it is possible to obtain tighter convex relaxations when convexifing a set directly rather than using convex envelopes of functions describing the set. In this paper, we pursue the convexification of sets that appear as substructures of general QCQPs.

A common approach to obtain convex relaxations of QCQPs is that of using semi-definite programming (SDP) relaxations. The first step in this approach is to write an equivalent form of the QCQP~(\ref{eq:qcqp}) as follows:
\begin{eqnarray}
&\textup{min}& \langle Q^0, X\rangle + (a^{0})^{\top}x  \label{eq:QCQP2a}\\
& \textup{s.t.}& \langle Q^k, X\rangle + (a^k)^{\top}x \leq b_k \qquad \forall k \in \{1, \dots, m\} \label{eq:QCQP2b}\\
&& \textup{rank}\left(\left[ \begin{array}{cc} 1 & x^{\top}\\ x & X \end{array}\right ]\right ) = 1 \label{eq:QCQP2c}\\
&& x \in [0, \ 1 ]^n,\label{eq:QCQP2d}
\end{eqnarray}
where $\langle U, V \rangle := \sum_{i = 1}^{n_1}\sum_{j =1}^{n_2} U_{ij} V_{ij}$. 
Observe that all the non-convexity of the problem is now captured by the rank-1 condition. This motivates our study of sets defined by a rank-1 
\footnote{With some abuse of terminology, we will be referring to matrices whose rank is at most one by simply rank-1 matrices.} constraint 
together with some linear side constraints:
\begin{eqnarray}\label{eq:cU}
\cU_{(n_1,n_2)}^m( [A^k, b_k]_{k=1}^m) :=
 \left\{ W \in \mathbb{R}^{n_1 \times n_2}_{+}\, | \, \langle A^k, W\rangle   \leq b_k, \forall k \in \{1, \dots, m\}, \ \rank(W) \leq 1 \right\}, 
\end{eqnarray} 
where we will recover  (\ref{eq:QCQP2b})--(\ref{eq:QCQP2d}) if we replace $W$ with $\left[ \begin{array}{cc} 1 & x^{\top}\\ x & X \end{array}\right]$ and with appropriate choice of $A^k$s.

The starting point of our investigations is the classical result of~\cite{Burer2009} that says: if the linear inequalities in (\ref{eq:cU}) are $W_{ij} \leq 1$ for all $i\in \{1, \dots, n_1\}$ and $j\in \{1, \dots, n_2\}$, then $\conv(\cU)$ is exactly the boolean quadric polytope~\cite{padberg1989boolean}. This is a well-studied polytope and inequalities describing this set, such as the McCormick inequalities~\cite{mccormick1976computability} and triangle inequalities~\cite{padberg1989boolean}, are already used in practice. The paper~\cite{bonami2016solving} shows the use of more complicated inequalities valid for the boolean quadric polytope for solving box-constrained quadratic programs.  

However, to the best our knowledge, no other particular choice of structured linear side constraint has ever been studied (see~\cite{bienstock2016outer} for intersection cuts for rank-1 sets with general linear constraints). 
Let us give a simple choice of linear side constraints as a motivating example (where we do not assume the matrix variable is a square matrix):  
bounds on rows or columns of the $W$ variable, i.e. the set
\begin{eqnarray}\label{eq:rowset}
{\cU}^{\textup{row}}_{(n_1,n_2)}(l,u):= \left\{ W \in \mathbb{R}^{n_1 \times n_2}_{+}\, | \, l_i \leq \sum_{j= 1}^{n_2}W_{ij} \leq u_i, \forall i \in \{1, \dots, n_1\}, \ \rank(W) \leq 1 \right\}, 
\end{eqnarray}
where we assume $0 \leq l \leq u$. This choice of side constraints is not arbitrary, but comes up  naturally for the pooling problem~\cite{haverly1978studies, Gupte2017}.  Also note that such a relaxation could always be constructed for any bounded QCQP. 

\subsection{Contributions} 
In this paper, we explore general conditions under which the convex hull of the set $\cU_{(n_1,n_2)}^m( [A^k, b_k]_{k=1}^m)$ is polyhedral or second-order cone representable (SOCr), and also show that in each of these cases a linear objective function can be optimized over these sets  in polynomial time. These results are presented in Section~\ref{sec:genResults}. It turns out that the set ${\cU}^{\textup{row}}_{(n_1,n_2)}(l,u)$ introduced in (\ref{eq:rowset}) is a special case of the sets studied in Section~\ref{sec:genResults} and  its convex hull is polyhedral.  

In Section~\ref{sec:poolResults},  we specialize the general results of Section~\ref{sec:genResults} for sets like ${\cU}^{\textup{row}}_{(n_1,n_2)}(l,u)$ that are applicable for the pooling problem. We present results on the polyhedrality of convex hull, valid inequalities and extended formulations, and complexity of separating inequalities in the original space. Specifically in Section~\ref{sec:mainpool}, we present several formulations (and related discretizations) of the generalized pooling problem (i.e. pooling problems that have pool-to-pool arcs). Then, we illustrate how the sets studied here appear as substructures in different ways in the different pooling formulations. 

Finally, in Section~\ref{sec:maincomp}, we present results from computational experiments, which show that the new inequalities generated from substructures similar to ${\cU}^{\textup{row}}_{(n_1,n_2)}(l,u)$,  help in  improving dual bounds significantly. 

\section{General results}
\label{sec:genResults}
\textbf{Notation:} We use $\conv(S)$ to denote convex hull of a set $S$, $\extr(S)$ to denote the set of extreme points of a set $S$, $\proj_x (S)$ to denote the projection of a set $S$ onto the $x$ variables, $[n]$ to denote the set $\{1, \dots, n\}$, and $\mathbb{R}^n_{++}$ to be the set of $n$ dimensional positive vectors. Let us also define the set of $m-$partitions of a set $S$ as 
\[
\cP_m(S) := \left \{(S_1, \dots, S_m)  \,|\,  \cup_{k=1}^m S_k = S, \ S_j \cap S_k = \emptyset \text{ for } j\neq k \right\}.
\]

In this section, we present cases in which $\conv \big ( \cU_{(n_1,n_2)}^m( [A^k, b_k]_{k=1}^m)  \big )$ is either polyhedral or second-order cone representable.

\subsection{Some cases with polyhedral representable convex hulls}
We start with the trivial case of a single linear side constraint.
\begin{proposition}\label{prop:singleSide}
Suppose the set \ $\cU_{(n_1,n_2)}^1 ( [A^1, b_1])$ defined as in \eqref{eq:cU} is bounded\footnote{Boundedness is equivalent to $A^1_{ij} > 0$ for all $i\in [n_1],\ j \in [n_2]$.}. Then, we have
\[
\conv\big (\cU_{(n_1,n_2)}^1( [A^1, b_1]) \big ):=
 \left\{ W \in \mathbb{R}^{n_1 \times n_2}_{+}\, | \, \langle A^1, W\rangle   \leq b_1 \right\}.
\]
\end{proposition}
\noindent The proof of Proposition \ref{prop:singleSide} follows from the fact that the extreme points of the set $$\left\{ W \in \mathbb{R}^{n_1 \times n_2}_{+}\, | \, \langle A^1, W\rangle   \leq b_1 \right\},$$ are rank-1 matrices. Therefore, convexifying with just one constraint is not very interesting. 

Next, we allow multiple linear side constraints with certain rank-1 constraint matrices.
\begin{theorem}\label{thm:multrank1}
Consider the set  $\cU_{(n_1,n_2)}^m ( [A^k, b_k]_{k=1}^m)$ defined in \eqref{eq:cU}  where the constraint matrices  are of the form
\[
A^k := \alpha^k \beta^{\top} \qquad k\in[m],
\]
where $\alpha^k \in \mathbb{R}^{n_1}$ for $k\in[m]$ and $\beta \in \mathbb{R}_{++}^{n_2}$. Moreover, let the $\alpha^k$'s be such that 
\begin{eqnarray}\label{eq:recbnd}
\{u \in \mathbb{R}^{n_1}_{+}\,|\, (\alpha^k)^{\top}u \leq 0, \forall \ k \in [m] \} = \{0\}.  
\end{eqnarray}
Then, the following hold:
\begin{enumerate}
\item
$
\conv \big (\cU_{(n_1,n_2)}^m(  [A^k, b_k]_{k=1}^m ) \big )
$
is a polyhedral set.
\item The set of extreme points of ${\cU}^{m}_{(n_1, n_2)} ( [A^k, b_k]_{k=1}^m)$ are of the form:
\begin{equation*}
\begin{split}
\extr \left({\cU}^{m}_{(n_1, n_2)} ( [A^k, b_k]_{k=1}^m)\right) &= \left\{  \gamma  e_h^\top  \,|\, h\in[n_2] \right\},
\end{split}
\end{equation*}
where $e_h$ is the $n_2$-dimensional vector with all components zero except the $h^{th}$  component which is $1$ and the $\gamma$'s are extreme points of the set:
$$\left\{ \gamma \in \mathbb{R}^{n_1}_{+}\,|\, \sum_{i = 1}^{n_1}\alpha^k_i \beta_h\gamma_i \leq b_k, \ k\in[m] \right\}.$$ 
\item A compact extended formulation of $\textup{conv}\left(\cU_{(n_1,n_2)}^m ( [A^k, b_k]_{k=1}^m)\right)$ is given by:
\begin{align}
\sum_{j = 1}^{n_2} t_j &= 1  \label{eq:extended1}\\
\sum_{i=1}^{n_1}\alpha_i^k\beta_jW_{ij} &\leq b_kt_j \ &\forall& k \in [m], j \in [n_2]\\
t_j &\geq 0 \ &\forall& j \in [n_2]. \label{eq:extended3}
\end{align}
Therefore, a linear function can be optimized in polynomial time on $\cU_{(n_1,n_2)}^m ( [A^k, b_k]_{k=1}^m)$.
\end{enumerate}
\end{theorem}
\begin{proof}
To simplify notation, we will just write $\cU_{(n_1,n_2)}^m$ instead of $\cU_{(n_1,n_2)}^m ( [A^k, b_k]_{k=1}^m)$ in this proof.

Observe first that condition (\ref{eq:recbnd}) together with the fact that $\beta_j >0$ for $j \in [n_2]$ imply that $\{W \in \mathbb{R}^{n_1 \times n_2}_{+} \,|\, \langle \alpha^k\beta^{\top}, W\rangle \leq 0 , \forall k \in [m]\}$ is bounded, and therefore $\cU_{(n_1,n_2)}^m$ is bounded. 
Therefore, to prove (i), it is sufficient to show that the set of extreme points is finite. We will begin by showing that in any extreme point of the set $\cU_{(n_1,n_2)}^m$, there is at most one non-zero column. By contradiction, let us assume that $\hat W$ is an extreme point with two non-zero columns. Since $\rank(\hat W) \leq 1$ and $\hat W \ge 0$, there exists two vectors $\hat x \in \mathbb{R}_+^{n_1}$ and $\hat y \in \mathbb{R}_{+}^{n_2}$ such that $\hat W = \hat x \hat y^{\top}$ (note we  may also assume $\hat x \in \mathbb{R}_{-}^{n_1}$ and $\hat y \in \mathbb{R}_{-}^{n_2}$; however we cannot have $\hat{x}$ and $\hat{y}$ with different component with different signs, since then $\hat W$ is not non-negative). Without loss of generality, let us assume that the first two components of $y$ are positive. Let us define the set $\mathcal{K}(W) := \{k: \ \langle A^k, W\rangle   = b_k \}$. Consider two new points $W^{\pm} = \hat x (\hat y \pm \epsilon_1 e_1 \mp \epsilon_2 e_2)^{\top}$. Observe that we can select small but positive $\epsilon_1$ and $\epsilon_2$ such that both $W^+$ and $W^-$ belong to  $\cU_{(n_1,n_2)}^m$ since  for each $k \in \mathcal{K}(\hat W)$ we have
\[
b_k =  \langle  \alpha^k \beta^{\top} ,   \hat x (\hat y \pm \epsilon_1 e_1 \mp \epsilon_2 e_2)^{\top} \rangle  \Leftrightarrow
(\hat x^{\top}\alpha^k) (\beta_1 \epsilon_1 - \beta_2 \epsilon_2) = 0 \Leftrightarrow \beta_1 \epsilon_1 = \beta_2 \epsilon_2,
\]
where note that $\beta_1 >0$ and $\beta_2 > 0$. Hence, we reach a contradiction to $\hat W$ being an extreme point. 

Finally, the fact that there is at most one non-zero column in an extreme point of  $\cU_{(n_1,n_2)}^m$, the extreme points with $j^{th}$ column being non-zero is of the form:
\begin{eqnarray}
W^u = 0, u \in [n_2]\setminus \{j\}, \ W \geq 0, \  \beta_j(\alpha^k)^{\top}W^j \leq b_k \  \forall k \in [m],
\end{eqnarray}
where $W^u$ is the $u^{th}$ column of $W$. Thus, there are finitely many extreme points. Hence, the result follows. This also proves (ii).


To prove (iii), 
let $V^{{m}}_{(n_1,n_2)}$ be the compact extended formulation defined in (\ref{eq:extended1})-(\ref{eq:extended3}).

$\textup{conv} \left({\cU}^{{m}}_{(n_1,n_2)}\right) \subseteq \textup{proj}_W\left(V^{{m}}_{(n_1,n_2)}\right)$: Part~(ii) lists all the extreme points of ${\cU}^{{m}}_{(n_1,n_2)}$. It is straightforward to check that these extreme points belong to $\textup{proj}_W\left(V^{{m}}_{(n_1,n_2)}\right)$.

$\textup{proj}_W\left(V^{{m}}_{(n_1,n_2)}\right) \subseteq \textup{conv} \left({\cU}^{m}_{(n_1,n_2)}\right)  $: We first claim that in any extreme point of $V^{{m}}_{(n_1,n_2)}$ exactly one $t$ variable is positive. In particular, let $(\hat{t}, \hat{W}) \in V^{m}_{(n_1,n_2)}$ and without loss of generality let $\hat{t}_1 \neq 0, \hat{t}_2 \neq 0$. Let $\epsilon$ such that $0 < \epsilon \leq \frac{1}{2}\cdot \textup{min}\{\hat{t}_1, \hat{t}_2\}$. Then we construct two solutions $(\tilde{t}, \tilde{W})$ and $(\bar{t}, \bar{W})$ as follows:
\begin{itemize}
\item[] $\tilde{t}_1 = \hat{t}_1 - \epsilon, \bar{t}_1 = \hat{t}_1 + \epsilon; \tilde{t}_2 = \hat{t}_1 + \epsilon, \bar{t}_1 = \hat{t}_1 - \epsilon; \tilde{t}_j = \bar{t}_j = \hat{t}_j \forall j \in [n_2]\setminus \{1,2\};$
\item[] For $j \in \{1,2\}$: $\tilde{W}^{ij} = \frac{\tilde{t}_j}{\hat{t}_j}\hat{W}_{ij}$ and $\bar{W}^{ij} = \frac{\bar{t}_j}{\hat{t}_j}\hat{W}_{ij}$; $\tilde{W}_{ij} = \bar{W}_{ij} = \hat{W}_{ij} \forall i \in [n_1], j \in [n_2]\setminus \{1,2\}$.
\end{itemize}
It is straighforward to verify that $(\tilde{t}, \tilde{W})$ and $(\bar{t}, \bar{W})$ belong to $V^{{m}}_{(n_1,n_2)}$. Thus, $(\hat{t}, \hat{W})$ is not an extreme point of $V^{{m}}_{(n_1,n_2)}$. 

Since an extreme point of $V^{{m}}_{(n_1,n_2)}$ has exactly one positive component in $t$, each extreme point of $V^{m}_{(n_1,n_2)}$ are of the form: $t_{j'} = 1$, and $W_{j'}$ is an extreme point of $$\left\{x \in \mathbb{R}^{n_1}_{+} \,|\, \sum_{i}\alpha^k_i \beta_{j'} x_i \leq b^k \right\},$$
for some $j'$ and $t_j = 0$, $W_{ij} = 0$ for all $j \in [n_2]\setminus j'$ ($t_j = 0$ implies $W_{ij} = 0$, due to condition (\ref{eq:recbnd})). Thus $\textup{proj}_{W}\left(\extr \left( V^{m}_{(n_1,n_2)}\right) \right) = \extr\left( {\cU}^{m}_{(n_1,n_2)}\right)$. Therefore, we obtain that $\textup{proj}_W\left(V^{m}_{(n_1,n_2)}\right) \subseteq \textup{conv} \left({\cU}^{m}_{(n_1,n_2)}\right)  $. 
\end{proof}

\subsection{Some cases with second-order cone representable convex hulls}
In this section, we study two cases in which the convex hull is SOCr but not necessary polyhedral.
We first consider sets with two arbitrary linear side constraints. 

\begin{theorem}\label{prop:twoSide}
Suppose that  the set \ $\cU_{(n_1,n_2)}^2 ( [A^k, b_k]_{k=1}^{2})$ as defined in \eqref{eq:cU} is bounded. Then, the convex hull of $ \cU_{(n_1,n_2)}^2 ( [A^k, b_k]_{k=1}^{2}) $ is SOCr. Moreover, a linear function can be optimized in polynomial time over $\cU_{(n_1,n_2)}^2 ( [A^k, b_k]_{k=1}^2)$. 
\end{theorem}
\begin{proof}
%
Note that the since $\cU_{(n_1,n_2)}^2 ( [A^k, b_k]_{k=1}^{2})$ is assumed to be bounded, we have that the convex hull of $\cU_{(n_1,n_2)}^2 ( [A^k, b_k]_{k=1}^{2})$ is the convex hull of its extreme points. 

We first claim that  in any extreme point, there are at most two non-zero columns.  By contradiction, let us assume that $\hat W$ is an extreme point with three non-zero columns. 
Since $\rank(\hat W) = 1$ and $\hat W \ge 0$, there exists two vectors $\hat x \in \mathbb{R}_+^{n_1}$ and $\hat y \in \mathbb{R}_{+}^{n_2}$ such that $\hat W = \hat x \hat y^{\top}$. Without loss of generality, let us assume that the first three components of $y$ are positive. 
Consider two new points $W^{\pm} = \hat x (\hat y \pm \epsilon_1 e_1 \pm \epsilon_2 e_2 \pm  \epsilon_3 e_3)^{\top}$. Note that  we can select a non-zero vector $(\epsilon_1, \epsilon_2, \epsilon_3)$ with small enough magnitude such that both $W^+$ and $W^-$ belong to  $\cU_{(n_1,n_2)}^2 ( [A^k, b_k]_{k=1}^{2})$ since we have
\[
\langle A^k, \hat x \hat y^{\top} \rangle =  \langle  A^k  ,   \hat x (\hat y \pm \epsilon_1 e_1 \pm \epsilon_2 e_2  \pm \epsilon_3 e_3)^{\top} \rangle   \Leftrightarrow
0 =   (\hat x^{\top} A^k e_1) \epsilon_1 +   (\hat x^{\top} A^k e_2) \epsilon_2  +  (\hat x^{\top} A^k e_3) \epsilon_3 \ \ k=1,2.
\]
 Note that this homogeneous system is guaranteed to have a non-trivial solution in $(\epsilon_1, \epsilon_2, \epsilon_3)$. Hence, we reach to a contradiction.
 
Following a similar procedure, one can also show that there are at most two non-zero rows in any extreme point. Therefore, we deduce that  the largest non-zero submatrix of the extreme point of the set $\cU_{(n_1,n_2)}^2 ( [A^k, b_k]_{k=1}^{2})$ is $2\times2$.

Now, let us fix a particular two-by-two submatrix of $W$. We will show that the convex hull of all extreme points whose support is on this particular choice of two-by-two submatrix is SOCr. Since the convex hull of the union of  SOCr compact sets is SOCr, we have that convex hull of the extreme points of $\cU_{(n_1,n_2)}^2 ( [A^k, b_k]_{k=1}^{2})$ is SOCr, which would complete the proof.  Without loss of generality, consider  the extreme points with support on the first two rows and first two columns. Then we are looking for the convex hull of the following set (which is satisfied by all the extreme points):
\begin{eqnarray}\label{eq:12rowcol}
\begin{array}{rcl}
\langle A^k, W \rangle  &\leq& b_k \ \forall k \in \{1, 2\} \\
W_{ij} &\geq& 0 \ \forall i \in [n_1], j \in [n_1]\\
W_{ij} &=& 0 \ \textup{ if } i \geq 3 \textup{ or } j \geq 3 \\
W_{11} W_{22} &=& W_{21} W_{12}. 
\end{array}
\end{eqnarray}

It was recently shown in~\cite{santana2018convex} that the convex hull of a set described by a quadratic constraint and  bounded polyhedral constraints is SOCr (note that since (\ref{eq:12rowcol}) represents a subset of $\cU_{(n_1,n_2)}^2 ( [A^k, b_k]_{k=1}^{2})$, it is bounded). This completes the proof. 

Finally, note that optimizing a linear function on $\cU_{(n_1,n_2)}^2 ( [A^k, b_k]_{k=1}^{2})$ is equivalent to optimizing a linear function on the set of its extreme points. Since there are ${{n_1} \choose {2}}\cdot {{n_2} \choose {2}}$ possible choices of supports and for each choice, we can optimize a linear function over (\ref{eq:12rowcol}) in polynomial time (the size of the SOC-representation is linear in the total number of faces of the polytope, which is fixed in (\ref{eq:12rowcol}) since there are only six linear inequalities in four variables).
\end{proof}

Note that the papers~\cite{Burer2017,modaresi2017convex} show that the convex hull of two quadratic constraints is SOCr.  Although Theorem~\ref{prop:twoSide} is of similar flavor,  it allows for the additional side constraints, namely, the non-negativities on the bilinear terms (i.e.,  the constraint $W \geq 0$). Moreover, the proof technique here is completely different from those used in~\cite{Burer2017,modaresi2017convex}. 

Next, we allow multiple linear side constraints with certain rank-2 constraint matrices and show that a result similar to Theorem~\ref{prop:twoSide} is possible. 
\begin{theorem}\label{prop:twoSideMultiple}
Let $n_1=n_2\ge3$ and suppose that the set \ $\cU_{(n_1,n_1)}^m ( [A^k, b_k]_{k=1}^m)$ defined in \eqref{eq:cU} is bounded. Assume that the constraint matrices  are of the form
\[
A^k := \alpha^k \beta\beta^{\top} + \gamma^k \delta \delta^{\top} \qquad k\in[m],
\]
where $\alpha^k,\gamma^k \in \mathbb{R}_{+}$ for $k\in[m]$ and $\beta,\delta \in \mathbb{R}_{++}^{n_1}$. 
Then, $\conv \big( \cU_{(n_1,n_1)}^m ( [A^k, b_k]_{k=1}^{m}) \big )$ is SOCr.  Moreover, for a fixed $m$, a linear function can be optimized in polynomial time over $\cU_{(n_1,n_2)}^m ( [A^k, b_k]_{k=1}^m)$. 
\end{theorem}
\begin{proof}
First,  note that the non-negativity of $\alpha^k, \gamma^k$ and positivity of $\beta, \delta$ imply that the set is bounded. 

Observe that the following system has a non-trivial solution in $\epsilon$
\begin{equation*}
\begin{bmatrix} 
\beta_{j_1} & \beta_{j_2} & \beta_{j_3} \\ 
\delta_{j_1} & \delta_{j_2} & \delta_{j_3} 
\end{bmatrix}
\begin{bmatrix} \epsilon_1 \\ \epsilon_2 \\ \epsilon_3 \end{bmatrix}
=\begin{bmatrix} 0\\0\end{bmatrix},
\end{equation*}
for every distinct indices $j_1,j_2,j_3$. This  guarantees that the following system has a non-trivial solution in $\epsilon$:
\begin{equation}\label{eq:auxSystem}
\begin{split}
\alpha^k (\beta^\top\hat x)[\beta_{j_1}\epsilon_1+ \beta_{j_2}\epsilon_2+ \beta_{j_3}\epsilon_3] &= 0 \quad k\in[m] \\
\gamma^k (\delta^\top\hat x)[\delta_{j_1}\epsilon_1+ \delta_{j_2}\epsilon_2+ \delta_{j_3}\epsilon_3] &= 0  \quad k\in[m].
\end{split}
\end{equation}

After establishing the above relation, we proceed similar to the proof of Proposition~\ref{prop:twoSide}.  We again claim that  in any extreme point, there are at most two non-zero columns.  By contradiction, let us assume that $\hat W = \hat x \hat y^\top$ is an extreme point with three non-zero columns. Without loss of generality, let these be the first three columns. However, in this case, the following system has a non-trivial solution in $\epsilon$,
\begin{equation*}
\begin{split}
\alpha^k (\beta^\top\hat x)[\beta_{1}\epsilon_1+ \beta_{2}\epsilon_2+ \beta_{3}\epsilon_3] +
\gamma^k (\delta^\top\hat x)[\delta_{1}\epsilon_1+ \delta_{2}\epsilon_2+ \delta_{3}\epsilon_3] = 0  \quad k\in[m],
\end{split}
\end{equation*}
due to \eqref{eq:auxSystem} (with the choice of $j_1=1, j_2=2, j_3=3$). Therefore, two points  defined as $W^{\pm} = \hat x (\hat y \pm \epsilon_1 e_1 \pm \epsilon_2 e_2 \pm  \epsilon_3 e_3)^{\top}$ belong to $\cU_{(n_1,n_1)}^m ( [A^k, b_k]_{k=1}^m)$ for some non-zero $\epsilon$, contradicting to the assumption that $\hat W$ is an extreme point.

Using the same procedure, one can also prove that at an extreme point of $\conv\big(\cU_{(n_1,n_1)}^m ( [A^k, b_k]_{k=1}^m)\big)$, there can be at most two non-zero rows. Therefore, we deduce that  the largest non-zero submatrix of an extreme point is $2\times2$. The rest of the proof is identical to the last part of the proof of Theorem~\ref{prop:twoSide}.
\end{proof}

\section{Application of results to the pooling problem}
\label{sec:poolResults}

In this section, we show how the results presented in Section~\ref{sec:genResults} apply to the pooling problem.

\subsection{Convex hull results and valid inequalities}\label{sec:mainconv}
Our first result of this section, Proposition~\ref{prop:row} presented below, is regarding the convex hull of the set ${\cU}^{\textup{row}}_{(n_1,n_2)}(l,u)$ introduced in (\ref{eq:rowset}). 
\begin{proposition}\label{prop:row}
Consider the set ${\cU}^{\textup{row}}_{(n_1,n_2)}(l,u)$ described in (\ref{eq:rowset})  where we assume $u_i$ is finite for every $i \in [n_1]$. Then, the following hold: 
\begin{enumerate}
\item $\textup{conv}\left({\cU}^{\textup{row}}_{(n_1,n_2)}(l,u)\right)$ is a polyhedral set. 
\item A compact extended formulation of $\textup{conv}\left({\cU}^{\textup{row}}_{(n_1,n_2)}(l,u)\right)$ is given by:
\begin{align}
\sum_{j = 1}^{n_2} t_j &= 1  \label{eq:ex1}\\
l_it_j &\leq W_{ij} \leq u_it_j \ &\forall& i \in [n_1], j \in [n_2]\\
t_j &\geq 0 \ &\forall& j \in [n_2]. \label{eq:ex3}
\end{align}
Therefore, a linear function can be optimized in polynomial time on ${\cU}^{\textup{row}}_{(n_1,n_2)}(l,u)$.
\item Let $\mathcal{I} := \{ i \in [n_1]\,|\, l_i > 0 \}$. Without loss of generality, let us assume that the first $|\cI|$ components of $l$ are positive.
Assuming $u_i > 0$ for all $i \in [n_1]$ (otherwise, we may fix all the variable in the corresponding row to be zero), $\textup{conv}\left({\cU}^{\textup{row}}_{(n_1,n_2)}(l,u)\right)$ is given by:
\begin{align}
\sum_{i = 1}^{n_1}\sum_{j \in T_i}\frac{W_{ij}}{u_i} &\leq  1 & \forall& (T_1, \dots, T_{n_1}) \in \cP_{n_1} ([n_2])
 \label{eq:rowconv1}\\
\sum_{i \in \mathcal{I}}\sum_{j \in T_i}\frac{W_{ij}}{l_i} & \geq 1 & \forall& (T_1, \dots, T_{|\mathcal{I}|})  \in \cP_{|\mathcal{I}|} ([n_2])
\label{eq:rowconv2}\\
l_{i_1} W_{i_2 j} & \leq  u_{i_2} W_{i_1 j} & \forall& j \in [n_2], i_1 \in \mathcal{I}, i_2 \in [n_1] \textup{ and }i_1 \neq i_2 \label{eq:rowconv3} \\
W_{ij} &\geq 0 \  &\forall& i \in [n_1], j \in [n_2], \label{eq:rowconv4}
\end{align}
where (\ref{eq:rowconv2}) are not required (in fact not well-defined) in the convex hull description if $\mathcal{I} = \emptyset$. The inequalities (\ref{eq:rowconv1}) and (\ref{eq:rowconv2}) can be separated in polynomial-time. 
\end{enumerate}
\end{proposition}

Note that parts (i) and (ii) of Proposition~\ref{prop:row} are a corollary of Theorem~\ref{thm:multrank1}, where we choose $\beta = e$ and, for  $i\in[n_1], \ k\in[2n_1]$, we define: $\alpha^k = e_i, \ b_k = u_i$ if $k=i$; $\alpha^k = -e_i, \ b_k = -l_i$ if $k=n_1+i$. For part (iii), we project the extended formulation presented in part (ii) to the space of the original variables $W$ via Fourier-Motzkin procedure. The details of our proof of (iii) is presented in Appendix~\ref{sec:proofprop3}.

We next highlight an interesting connection between the extended formulation presented in Proposition~\ref{prop:row} and McCormick inequalities. In particular, note that since $W$ is a rank-1 matrix, we have that $\frac{W_{ij}}{\sum_{j'} W_{ij'}}$ is a constant independent of the row index $i$ (assuming $\sum_{j'} W_{ij'} \neq 0$). More formally, we may write:
\begin{eqnarray}\label{eq:preextend}
W_{ij} = t_j\left(\sum_{j' = 1}^{n_2}W_{ij'}\right) \qquad \forall i \in [n_1],  j \in [n_2],
\end{eqnarray}
where the fact that the ratio variable $t$ is independent of row index, is indicated with it being indexed only with column index $j$. Note that it is straightforward to see that the $t$ variables satisfy:
\begin{eqnarray}\label{eq:partext}
\sum_{j = 1}^{n_2}t_j = 1, \ t_j \geq 0\qquad \forall j \in [n_2] .	
\end{eqnarray} 
Finally, we may apply McCormick inequalities for (\ref{eq:preextend}), by observing that we have the bounds $0 \leq t_j \leq 1$ and $l_i \leq \sum_{j' = 1}^{n_2}W_{ij'} \leq u_i$, to obtain:
\begin{eqnarray}
 l_i t_j  \le  &W_{ij} &\leq u_it_j \label{eq:MCC12}\\
 u_it_j + \sum_{j' = 1}^{n_2}W_{ij'} - u_i  \leq &W_{ij} & \leq l_it_j + \sum_{j' = 1}^{n_2}W_{ij'} - l_i\label{eq:MCC34}.
\end{eqnarray}
Now, observe that (\ref{eq:partext}) together with the first two inequalities of the McCormick envelopes (\ref{eq:MCC12}) 
above yields the extended formulation. Indeed, one can also check that (\ref{eq:MCC34}) 
is  implied by (\ref{eq:MCC12}) and (\ref{eq:partext}). This connection between McCormick inequalities and the extended formulation of the convex hull will be used later  when we discretize various substructures of the pooling problem (see Section~\ref{subsec:discret}). 

Next, we record another simple application of Theorem~\ref{thm:multrank1} that will be useful for the pooling problem. Consider the set
\begin{eqnarray}\label{eq:rowplusset}
\begin{array}{rl}
{\cU}^{\textup{row}+}_{(n_1,n_2)}(l,u, L, U):= &\{ W \in \mathbb{R}^{n_1 \times n_2}_{+}\, | \\
&\, l_i \leq \sum_{j= 1}^{n_2}W_{ij} \leq u_i, \qquad\forall i \in [n_1]\} \\
&L \leq \sum_{i=1}^{n_1}\sum_{j = 1}^{n_2}W_{ij} \leq U \\
& \textup{rank}(W) \leq 1\}, 
\end{array}
\end{eqnarray}
where we assume $0 < u_i \leq U$ (otherwise we can replace $u_i$ by $U$) and $U \leq \sum_{i =1}^{n_1}u_i$ (otherwise we replace $U$ by $\sum_{i =1}^{n_1}u_i$).

\begin{proposition}\label{prop:rowplus}
Consider the set  ${\cU}^{\textup{row}+}_{(n_1,n_2)}(l,u, L, U)$ described in \eqref{eq:rowplusset}  where we assume that $U$ is finite. Then, the following hold:
\begin{enumerate}
\item
An extended formulation of $\conv\left({\cU}^{\textup{row}+}_{(n_1,n_2)}(l,u, L, U)\right)$ is given by:
\begin{align}
\sum_{j = 1}^{n_2} t_j &= 1 \\
Lt_j &\leq \sum_{i = 1}^{n_1}W_{ij} \leq Ut_j  &\forall& j \in [n_2]\\
l_it_j &\leq W_{ij} \leq u_it_j \ &\forall& i \in [n_1], \forall j \in [n_2]\\
t_j &\geq 0 \ &\forall& j \in [n_2].
\end{align}

\item
Let $\mathcal{I} := \{ i \in [n_1]\,|\, l_i > 0 \}$. Without loss of generality, let us assume that the first $|\cI|$ components of $l$ are positive.
Assuming $u_i > 0$ for all $i \in [n_1]$, 
$\textup{conv}\left({\cU}^{\textup{row+}}_{(n_1,n_2)}(l,u,L,U)\right)$ is given by:
\begin{align}
\sum_{i = 1}^{n_1}\sum_{j \in T_i}\frac{W_{ij}}{u_i} + \frac1U \sum_{i = 1}^{n_1}\sum_{j \in T_0}  W_{ij} &\leq  1 & \forall& (T_0, T_1, \dots, T_{n_1}) \in \cP_{n_1} ([n_2])
 \label{eq:rowplusconv1}\\
\sum_{i \in \mathcal{I}}\sum_{j \in T_i}\frac{W_{ij}}{l_i}  + \frac1L \sum_{i \in \mathcal{I}}\sum_{j \in T_0}  W_{ij}  & \geq 1 & \forall& (T_0, T_1, \dots, T_{|\mathcal{I}|})  \in \cP_{|\mathcal{I}|} ([n_2])
\label{eq:rowplusconv2}\\
\eqref{eq:rowconv3} - \eqref{eq:rowconv4}. \notag
\end{align}
The inequalities (\ref{eq:rowplusconv1}) and (\ref{eq:rowplusconv2}) can be separated in polynomial-time. 
\end{enumerate}
\end{proposition}

Part (i) of Proposition~\ref{prop:rowplus} is a corollary of Theorem~\ref{thm:multrank1}. It can also be viewed as an extension of Proposition~\ref{prop:row}(ii), where we further define $\alpha^k=e, \ b^k=U$ if $k=2n_1+1$; and $\alpha^k=-e, \ b^k=-L$ if $k=2n_1+2$. The proof of part (ii) is presented in Appendix~\ref{sec:proofprop4}.

Clearly, all the results above also apply to a set where instead of row bounds, we have column bounds, i.e. the set:
\begin{eqnarray}\label{eq:colset}
{\cU}^{\textup{col}}_{(n_1,n_2)} (l', u'):= \left\{ W \in \mathbb{R}^{n_1 \times n_2}_{+}\, | \,  l_j' \leq \sum_{i= 1}^{n_1}W_{ij} \leq  u_j', \forall j \in [n_2], \textup{rank}(W) \leq 1 \right\},
\end{eqnarray}
and the set ${\cU}^{\textup{col}+}_{(n_1,n_2)}(l',u', L, U)$ defined  analogously to ${\cU}^{\textup{row}+}_{(n_1,n_2)}(l,u, L, U)$.
A natural extension to these results is the intersection of the two sets. The first observation is that convex hull is not polyhedral.
\begin{theorem}\label{thm:2by2}
$\textup{conv} \left({\cU}^{\textup{row}}_{(2,2)}(l,u) \cap {\cU}^{\textup{col}}_{(2,2)} (l', u')\right)$ is SOCr, but not polyhedral.
\end{theorem}
A proof of Theorem~\ref{thm:2by2} is presented in Appendix~\ref{sec:proofthm4}.
Unfortunately, we do not have a complete description of $\textup{conv}({\cU}^{\textup{row}}_{(n_1,n_2)}(l,u) \cap {\cU}^{\textup{col}}_{(n_1, n_2)}(l', u'))$ for general $n_1, n_2 \in \mathbb{Z}_{+}.$ 
We conjecture that optimizing a linear function on ${\cU}^{\textup{row}}_{(n_1,n_2)}(l,u) \cap {\cU}^{\textup{col}}_{(n_1, n_2)}(l', u')$ is NP-hard.

\subsection{Relaxations and restrictions by discretization}\label{subsec:discret}

We now provide inner and outer-approximations of the set ${\cU}^{\textup{row}}_{(n_1,n_2)}(l,u)$ via discretization. See \cite{MISENER2011876, Gupte2017,DeyGupte2015} for more details and examples of application of this technique to the pooling problem. Let us start with the outer-approximation (or relaxation) and the discretization of the $t_j$ variable as 
\begin{equation*}
t_j = \sum_{h=1}^H 2^{-h} z_{jh} + \gamma_j,
\end{equation*}
where $H \in \mathbb{Z}_{++}$ is the discretization level, $z_{jh}$ are binary variables and $\gamma_j$ is a continuous  non-negative variable  upper-bounded by $2^{-H}$. As discussed in the previous section, applying McCormick inequalities to \eqref{eq:preextend} yields the convex hull of ${\cU}^{\textup{row}}_{(n_1,n_2)}(l,u)$. Thus it makes sense to discretize $t$ in this equation, i.e. we can now rewrite \eqref{eq:preextend}    as
\begin{equation*}
W_{ij} = \left(\sum_{j' = 1}^{n_2}W_{ij'}\right) \left( \sum_{h=1}^H 2^{-h}  z_{jh} + \gamma_j \right ) 
 \qquad \forall i \in [n_1], \forall j \in [n_2].
\end{equation*}
Let us define $\alpha_{ijh}:= \left( \sum_{j' = 1}^{n_2}W_{ij'} \right) z_{jh}$ and $\beta_{ij}:= \left( \sum_{j' = 1}^{n_2}W_{ij'} \right) \gamma_{j}$, and write down the McCormick envelopes respectively as 
\begin{align}
 l_i z_{jh}  \le   & \alpha_{ijh} \leq u_i z_{jh} &  \forall i\in[n_1],  \forall j \in [n_2],  \forall h\in[H]  \label{eq:MCC12alpha}  \\
 u_i z_{jh} + \sum_{j' = 1}^{n_2}W_{ij'} - u_i  \leq &  \alpha_{ijh}  \leq l_i z_{jh} + \sum_{j' = 1}^{n_2}W_{ij'} - l_i   &   \forall  i\in[n_1],  \forall j \in [n_2],  \forall h\in[H] \label{eq:MCC34alpha},
\end{align}
and
\begin{align}
 l_i \gamma_j  \le  &\beta_{ij}  \leq u_i \gamma_j   &  \forall  i\in[n_1],  \forall j \in [n_2] \label{eq:MCC12beta}\\
 u_i \gamma_j + 2^{-H} \sum_{j' = 1}^{n_2}W_{ij'} - 2^{-H} u_i  \leq & \beta_{ij}  \leq l_i \gamma_j+ 2^{-H} \sum_{j' = 1}^{n_2}W_{ij'} -2^{-H} l_i&  \forall  i\in[n_1],  \forall j \in [n_2]  \label{eq:MCC34beta}.
\end{align}
Then, we obtain the  following outer-approximation of ${\cU}^{\textup{row}}_{(n_1,n_2)}(l,u)$:
\begin{eqnarray*}
\begin{array}{rl}
{\overline \cD}^{\textup{row}}_{(n_1,n_2,H)}(l,u) :=& \{ W \in \mathbb{R}_+^{n_1\times n_2}  \,|\,  \exists (\alpha , \beta, \gamma, z) \in 
\mathbb{R}^{n_1\times n_2 \times H} \times \mathbb{R}^{n_1\times n_2} \times \mathbb{R}^{H} \times \{0,1\}^{n_2 \times H}  | \\
&\,\eqref{eq:MCC12alpha} - \eqref{eq:MCC34beta},  \\
&  l_i \leq \sum_{j= 1}^{n_2}W_{ij} \leq u_i, \qquad\forall i \in [n_1], \\
& W_{ij} = \sum_{h=1}^H 2^{-h} \alpha_{ijh} + \beta_{ij}  , \qquad\forall i \in [n_1], \forall j \in [n_2] \}.
\end{array}
\end{eqnarray*}

\begin{proposition}\label{prop:D out}
For any $H \in  \mathbb{Z}_{++}$, we have ${\cU}^{\textup{row}}_{(n_1,n_2)}(l,u)  \subseteq {\overline \cD}^{\textup{row}}_{(n_1,n_2,H)}(l,u) $.
\end{proposition}
\noindent We will later use the set ${\overline \cD}^{\textup{row}}_{(n_1,n_2,H)}(l,u) $ to obtain MILP relaxations of the pooling problem.

Now, let us continue with the inner-approximation (or restriction) and discretize the $t_j$ variable as 
\begin{equation*}
t_j =  \frac{ 2^{H} }{2^H-1} \sum_{h=1}^H 2^{-h}  z_{jh}.
\end{equation*}
This time, equation \eqref{eq:preextend}  is rewritten  as
\begin{equation*}
W_{ij} = \left(\sum_{j' = 1}^{n_2}W_{ij'}\right) \left( \frac{ 2^{H} }{2^H-1}  \sum_{h=1}^H 2^{-h}  z_{jh}  \right ) 
 \qquad \forall i \in [n_1], \forall j \in [n_2].
\end{equation*}
Finally, we obtain the  following inner-approximation of ${\cU}^{\textup{row}}_{(n_1,n_2)}(l,u)$:
\begin{eqnarray*}
\begin{array}{rl}
{\underline \cD}^{\textup{row}}_{(n_1,n_2,H)}(l,u) :=& \{ W \in \mathbb{R}_+^{n_1\times n_2}  \,|\,  \exists (\alpha , z) \in 
\mathbb{R}^{n_1\times n_2 \times H} \times \mathbb{R}^{n_1\times n_2} \times \mathbb{R}^{H} \times \{0,1\}^{n_2 \times H}  | \\
&\,\eqref{eq:MCC12alpha} - \eqref{eq:MCC34alpha},  \\
&  l_i \leq \sum_{j= 1}^{n_2}W_{ij} \leq u_i, \qquad\forall i \in [n_1], \\
& W_{ij} = \frac{ 2^{H} }{2^H-1} \sum_{h=1}^H 2^{-h} \alpha_{ijh}   , \qquad\forall i \in [n_1], \forall j \in [n_2] \}.
\end{array}
\end{eqnarray*}

\begin{proposition}\label{prop:D in}
For any $H \in  \mathbb{Z}_{++}$, we have ${\underline \cD}^{\textup{row}}_{(n_1,n_2,H)}(l,u) \subseteq {\cU}^{\textup{row}}_{(n_1,n_2)}(l,u)   $.
\end{proposition}
\noindent We will later use the set ${\underline \cD}^{\textup{row}}_{(n_1,n_2,H)}(l,u) $ to obtain MILP restrictions of the pooling problem.

One can analogously define similar inner and outer-approximations, denoted respectively as  ${\underline \cD}^{\textup{col}}_{(n_1,n_2,H)}(l', u') $ and ${\overline \cD}^{\textup{col}}_{(n_1,n_2,H)}( l', u') $, for the set ${\cU}^{\textup{col}}_{(n_1,n_2)}( l', u') $ as well.

\subsection{Application of convex hull results to the pooling problem}\label{sec:mainpool}

Pooling problems constitute an important class of non-convex optimization problems in chemical engineering and process design. In this section, we will first review the \textit{multi-commodity flow formulation} of the generalized pooling problem, and then explain how the convex hull results in Section~\ref{sec:mainconv} can be applied to obtain strong relaxations.

Let us first introduce our notation of the generalized pooling problem, which primarily coincides with~\cite{alfaki2013multi}. Let $G=(N,A)$ be a graph with the node set $N$ and the arc set $A$. We will denote the set of  source (or input), intermediate (or pool) and  terminal (or output) nodes as $S$, $I$ and $T$, respectively. We have that $N=S\cup I \cup T$ and $A \subseteq (S \times (I\cup T)) \cup (I \times (I \cup T))$. 

We will denote the set of source nodes from which there is a path (not necessarily direct) to node $i$ as $S_i$, and the set of terminal nodes to which there is a path (not necessarily direct) from node $i$ as $T_i$. We also define
\[
N_i^+ := \{j  \,|\,  (i,j)\in A\} \ \text{ and } \ N_i^-:=\{j  \,|\,  (j,i) \in A\}.
\]

Let $K$ be the set of specifications tracked. Assume that the specification $k$ of source $s$ is given as $\lambda_{k}^s$ and the desired  specification $k$ of terminal $t$ should be in the interval $[\underline \mu_{k}^t, \overline \mu_{k}^t]$. The capacities of a node $i \in N$ and an arc $(i,j)\in A$ are denoted as  $U_i$ and $u_{ij}$ (we will also assume possibly trivial lower bounds for each node $i$ and arc $(i,j)$ as $L_i$ and $l_{ij}$).


Note that in the standard pooling problem there are no arcs from the set $(I \times (I \cup T))$. While the standard pooling problem is NP-hard to solve~\cite{alfaki2013strong,haugland2016computational}, there are also some positive results. For instance, a polynomial-time $n$-approximation algorithm ($n$ is the number of output nodes) is presented in~\cite{DeyGupte2015}, convex hull of special substructure (with one pool node) has been studied~\cite{luedtke2018strong}, and recently some very special cases of the pooling problem has been shown to be polynomially solvable~\cite{haugland2016computational,haugland2016pooling, boland2017polynomially,baltean2018piecewise}. However, none of these positive results apply for the generalized pooling problem in which the corresponding graph has arcs from the set $(I \times (I \cup T))$. It is well-understood that the generalized pooling problem is more challenging (and realistic) 
 than the standard pooling problem.

In the remainder of this section, we  present different ways to formulate the feasible region of the generalized pooling problem \emph{from the rank-1 perspective}. We will assume that the objective function is linear in terms of decision variables and specify its exact expression when we discuss the computational experiments since the objective function changes from instance to instance. Then, in Section~\ref{sec:comparison}, we compare all the relaxations introduced here with the ones known from literature.

\subsubsection{Source-based rank formulation}
\label{sec:sBasedForm}
Let us define the following two sets of decision variables: Let $f_{ij}$ be the amount of flow from node~$i$ to node $j$ and $x_{ij}^s$ be the amount of flow on arc $(i,j)$ \textit{originated} at the source $s \in S_i$. 

We will now describe the constraints of the generalized pooling problem. First, we start with the flow-related constraints:
\begin{align}
&L_i \le \sum_{j \in N_i^-} f_{ji} \le U_i & &\forall i \in I \cup T \label{eq:in capacity} \\
&L_i \le \sum_{j \in N_i^+} f_{ij} \le U_i & &\forall i \in S \label{eq:out capacity} \\
&l_{ij} \le f_{ij} \le u_{ij} & & \forall (i,j) \in A \cup \{(s,i):\ i\in I, \ s\in S_i\} \label{eq:arc capacity s} \\
&\sum_{j\in N_i^-} x_{ji}^s = \sum_{j\in N_i^+} x_{ij}^s & & \forall i\in I, \forall s\in S_i \label{eq:flow conserve s} \\
&\sum_{s \in S_i} x_{ij}^s = f_{ij} & & \forall (i,j) \in A \label{eq:flow ij def s} \\
&\sum_{j\in N_i^+} x_{ij}^s = f_{si} & & \forall i \in I, \forall s \in S_i \label{eq:flow si def} \\
&  x_{ij}^s \ge 0 & & \forall (i,j) \in A, \forall s \in S_i .
\end{align}

Here, constraints \eqref{eq:in capacity} and \eqref{eq:out capacity} correspond to node capacity while constraint \eqref{eq:arc capacity s} corresponds to the capacity of arc $(i,j)$. We note that there may not exist an arc between each $(s,i)$ pair, in which case we will call the quantity $f_{si}$ a \textit{ghost} flow. Constraint \eqref{eq:flow conserve s} is a flow conservation constraint for a node $i$ and source $s$ while constraints \eqref{eq:flow ij def s} and  \eqref{eq:flow si def} guarantees that the decomposed flow based on the origin sums up to the actual flow for each arc $(i,j)$. 

Next, we give the constraints that guarantee the specification requirements at the terminal nodes:
\begin{align}\label{eq:spec s}
\underline{\mu}_{k}^t  \sum_{j\in N_t^-} f_{jt}  \le
 \sum_{j\in N_t^-} \sum_{s \in S_j} \lambda_{k}^s x_{jt}^s   \le
\overline{\mu}_{k}^t  \sum_{j\in N_t^-} f_{jt}     \qquad \forall t\in T, \forall k\in K.
\end{align}

Finally, we present different ways to formulate the non-convex constraints of the pooling problem. A standard approach in the literature, which leads to the well-known \textit{$pq-$Formulation} \cite{tawarmalani2002convexification, alfaki2013multi}, is to define the proportion variables $q_i^s$ representing the fraction of flow at pool $i$ originated at source $s$, and include the following bilinear constraints:
\begin{equation}\label{eq:pq bilinear}
x_{ij}^s = q_i^s f_{ij}  \qquad \forall (i,j) \in A, \forall s  \in S_i .
\end{equation}
Our key observation is to rewrite the bilinear constraints in \eqref{eq:pq bilinear} as a set of rank restrictions on a matrix consisting of the decomposed flow variables  $ x_{ij}^s $ as follows:
\begin{equation}\label{eq:rank s}
\rank \left( \begin{bmatrix} x_{ij}^s \end{bmatrix}_{(s,j)\in S_i\times N_i^+} \right ) = 1 \qquad  \forall i\in I.
\end{equation}
The consequences of this realization will be discussed in detail when we construct the polyhedral relaxations of the pooling problem.

{\bf Polyhedral relaxations}: We will now present two ways to convexify the source-based formulation of the pooling problem. The first relaxation is obtained by the convexification of the rank constraints \eqref{eq:rank s} via the column-wise extended formulation, and defined over the following polyhedral set:
\begin{equation*}
\mathcal{F}_1^S :=
\left\{
(f,x)  \,|\,  \eqref{eq:arc capacity s}-\eqref{eq:spec s}, \ \begin{bmatrix} x_{ij}^s \end{bmatrix}_{(s,j)} \in  
\conv\left({\cU}^{\textup{col}+}_{(|S_i|, |N_i^+|)} \left ( \begin{bmatrix} l_{ij} \end{bmatrix}_{j} , \begin{bmatrix} u_{ij} \end{bmatrix}_{j}, L_{i} , U_{i} \right)\right), \forall i\in I
\right\}.
\end{equation*} 
We remark that this relaxation is equivalent to the McCormick relaxation of the $pq$-Formulation in which  bilinear equations \eqref{eq:pq bilinear} are convexified via the McCormick envelopes and the following implied constraints are added~\cite{adhya1999lagrangian, tawarmalani2002convexification}: $\sum_{s\in S_i}q^s_i = 1, \forall i \in I, \ L_i q_i^s \leq \sum_{j\in N_i^+}x_{ij}^s \leq U_i q_i^s, \forall i\in I, \forall s\in S_i$.

The second relaxation is a strengthening of the previous one with the addition of the row-wise extended formulation of the rank constraints  \eqref{eq:rank s}, and defined as below:
\begin{equation*}
\mathcal{F}_2^S :=
\left\{
(f,x) \in \mathcal{F}_1^S  \,|\,  \begin{bmatrix} x_{ij}^s \end{bmatrix}_{(s,j)} \in  
\conv\left({\cU}^{\textup{row}+}_{(|S_i|, |N_i^+|)} \left ( \begin{bmatrix} l_{si} \end{bmatrix}_{s} , \begin{bmatrix} u_{si} \end{bmatrix}_{s}, L_{i} , U_{i} \right)\right), \forall i\in I
\right\}.
\end{equation*} 
{We also considered a similar relaxation to $\mathcal{F}_1^S $, in which the row-wise (instead of the column-wise) extended formulation is used. However, we empirically observed that such a relaxation  is consistently weaker and hence omitted from further discussion.}



{\bf Discretization relaxations}: We will now present three ways to relax the source-based formulation of the pooling problem to obtain dual bounds using discretization techniques. The first relaxation is obtained by the discretization of the rank constraints \eqref{eq:rank s} via the column-wise extended formulation, and defined as follows:
\begin{align*}
\cM_1^S(H) := 
\left \{ (f,x)\in \mathcal{F}_1^S  \,|\,   \begin{bmatrix} x_{ij}^s \end{bmatrix}_{(s,j)} \in  
\conv\left({\overline \cD}^{\textup{col}}_{(|S_i|, |N_i^+|, H)} \left ( \begin{bmatrix} l_{ij} \end{bmatrix}_{j} , \begin{bmatrix} u_{ij} \end{bmatrix}_{j} \right)\right), \forall i\in I   \right \}.
\end{align*}
Here, $H\in\mathbb{Z}_+$ defines the discretization level. The other two relaxations are similarly defined as follows:
\begin{align*}
\cM_2^S(H) := 
\left \{ (f,x)\in \mathcal{F}_2^S  \,|\,   \begin{bmatrix} x_{ij}^s \end{bmatrix}_{(s,j)} \in  
\conv\left({\overline \cD}^{\textup{col}}_{(|S_i|, |N_i^+|, H)} \left ( \begin{bmatrix} l_{ij} \end{bmatrix}_{j} , \begin{bmatrix} u_{ij} \end{bmatrix}_{j} \right)\right), \forall i\in I   \right \},\\
\cM_3^S(H) := 
\left \{ (f,x)\in \mathcal{F}_2^S  \,|\,   \begin{bmatrix} x_{ij}^s \end{bmatrix}_{(s,j)} \in  
\conv\left({\overline \cD}^{\textup{row}}_{(|S_i|, |N_i^+|, H)} \left ( \begin{bmatrix} l_{si} \end{bmatrix}_{s} , \begin{bmatrix} u_{si} \end{bmatrix}_{s} \right)\right), \forall i\in I   \right \}.
\end{align*}

We do not consider simultaneously discretization with respect to row-wise and column-wise extended formulation, i.e. ${\overline \cD}^{\textup{col}}_{(|S_i|, |N_i^+|, H)} \left ( \begin{bmatrix} l_{ij} \end{bmatrix}_{j} , \begin{bmatrix} u_{ij} \end{bmatrix}_{j} \right) \cap {\overline \cD}^{\textup{row}}_{(|S_i|, |N_i^+|, H)} \left ( \begin{bmatrix} l_{si} \end{bmatrix}_{s} , \begin{bmatrix} u_{si} \end{bmatrix}_{s} \right) $, since this becomes a very large formulation.

{Following the definition of ${\overline \cD}^{\textup{col}}$, we conclude that $\cM_1^S(H)$ and $\cM_2^S(H)$ are obtained from discretizing the ratio variables $q^s_i$ in (\ref{eq:pq bilinear}) over $\mathcal{F}_1^S$ and $\mathcal{F}_2^S$, respectively. To interpret $\cM_3^S(H)$, we first observe that we can rewrite (\ref{eq:rank s}) as}
\begin{equation}\label{eq:pq bilinear output}
x_{ij}^s = f_{si} q_{ij}  \qquad \forall (i,j) \in A, \forall s  \in S_i ,
\end{equation}
where we have introduced the artificial ratio variables $q_{ij}$. Following the definition of ${\overline \cD}^{\textup{row}}$, we conclude that $\cM_3^S(H)$ is obtained from discretizing these newly introduced ratio variables $q_{ij}$ over $\mathcal{F}_2^S$.

{\bf Discretization restrictions}: We will now present two ways to restrict the source-based formulation of the pooling problem to obtain primal bounds. The first restriction is obtained via the discretization of the rank constraints \eqref{eq:rank s}, and defined as follows:
\begin{equation*}
\cG^{S}_1(H) :=
\left\{
(f,x)  \,|\,  \eqref{eq:arc capacity s}-\eqref{eq:spec s}, \ \begin{bmatrix} x_{ij}^s \end{bmatrix}_{(s,j)} \in  
 {\underline \cD}^{\textup{col}}_{(|S_i|, |N_i^+|,H)} \left ( \begin{bmatrix} l_{ij} \end{bmatrix}_{j} , \begin{bmatrix} u_{ij} \end{bmatrix}_{j}, L_{i} , U_{i} \right), \forall i\in I
\right\},
\end{equation*} 
where $H\in\mathbb{Z}_+$ defines the discretization level. We remark that this restriction is equivalent to discretizing the $q$ variables in each bilinear equation \eqref{eq:pq bilinear} of the $pq$-Formulation.

The second restriction is obtained in a similar way and given below:
\begin{equation*}
\cG^{S}_2(H) :=
\left\{
(f,x)  \,|\,  \eqref{eq:arc capacity s}-\eqref{eq:spec s}, \ \begin{bmatrix} x_{ij}^s \end{bmatrix}_{(s,j)} \in  
 {\underline \cD}^{\textup{row}}_{(|S_i|, |N_i^+|,H)} \left ( \begin{bmatrix} l_{si} \end{bmatrix}_{s} , \begin{bmatrix} u_{si} \end{bmatrix}_{s}, L_{i} , U_{i} \right), \forall i\in I
\right\}.
\end{equation*} 

\subsubsection{Terminal-based rank formulation}
\label{sec:tBasedForm}
Let us define a new set of decision variables $x_{ij}^t$ to denote the amount of flow on arc $(i,j)$ \textit{ended} at the terminal $t \in T_j$. 

Firstly, we again describe the flow constraints:
\begin{align}
&L_i \le \sum_{j \in N_i^-} f_{ji} \le U_i & &\forall i \in I \cup T \label{eq:in capacity terminal} \\
&L_i \le \sum_{j \in N_i^+} f_{ij} \le U_i & &\forall i \in S \label{eq:out capacity terminal} \\
&l_{ij} \le f_{ij} \le u_{ij} & & \forall (i,j) \in A \cup \{(j,t):\ j\in I, \ t\in T_j\} \label{eq:arc capacity t} \\
&\sum_{i\in N_j^-} x_{ij}^t = \sum_{i\in N_j^+} x_{ji}^t & & \forall j\in  I, \forall t\in T_j \label{eq:flow conserve t} \\
&\sum_{t \in T_j} x_{ij}^t = f_{ij} & & \forall (i,j) \in A \label{eq:flow ij def t} \\
&\sum_{i\in N_j^-} x_{ij}^t = f_{jt} & & \forall j \in I, \forall t \in T_j \label{eq:flow it def} \\
&  x_{ij}^t \ge 0 & & \forall (i,j) \in A, \forall t \in T_j .
\end{align}

Secondly, we give the constraints corresponding to the specification requirements:
\begin{align}\label{eq:spec t}
\underline{\mu}_{k}^t  \sum_{j\in N_t^-} f_{jt}  \le
  \sum_{(s,j)\in A: s\in S,t\in T_j } \lambda_k^s x_{sj}^t   \le
\overline{\mu}_{k}^t  \sum_{j\in N_t^-} f_{jt}   \qquad \forall t\in T, \forall k\in K.
\end{align}

To obtain a correct formulation of the pooling problem, the following bilinear equations can be added:
\begin{equation}\label{eq:tp bilinear}
x_{ij}^t = q_j^t f_{ij}  \qquad \forall (i,j) \in A, \forall t  \in T_j.
\end{equation}
This formulation, called the \textit{$tp-$Formulation}, was first proposed in \cite{alfaki2013strong} for the standard pooling problem, and then extended to the generalized pooling in \cite{boland2016new}.

We will again rewrite the bilinear constraints in \eqref{eq:tp bilinear} as a set of rank restrictions on a matrix consisting of the decomposed flow variables  $ x_{ij}^t $ as follows:
\begin{equation}\label{eq:rank t}
\rank \left( \begin{bmatrix} x_{ij}^t \end{bmatrix}_{(i,t)\in N_j^- \times T_j} \right ) = 1 \qquad  \forall j\in I.
\end{equation}

{\bf Polyhedral relaxations}: Similar to the relaxations of the source-based formulation in Section~\ref{sec:sBasedForm}, we define the following two polyhedral relaxations based on the terminal formulation:
\begin{align*}
\mathcal{F}_1^T&:=
\left\{
(f,x)  \,|\,  \eqref{eq:arc capacity t}-\eqref{eq:spec t}, \ \begin{bmatrix} x_{ij}^t \end{bmatrix}_{(i,t)} \in  
\conv\left({\cU}^{\textup{row}+}_{(|N_j^-|, |T_j|)} \left ( \begin{bmatrix} l_{ij} \end{bmatrix}_{i} , \begin{bmatrix} u_{ij} \end{bmatrix}_{i}, L_{j} , U_{j} \right)\right), \forall j \in I
\right\}, \\
\mathcal{F}_2^T &:=
\left\{
(f,x) \in\mathcal{F}_1^T   \,|\,   \begin{bmatrix} x_{ij}^t \end{bmatrix}_{(i,t)} \in 
\conv\left({\cU}^{\textup{col}+}_{(|N_j^-|, |T_j|)}  \left ( \begin{bmatrix} l_{jt} \end{bmatrix}_{t} , \begin{bmatrix} u_{jt} \end{bmatrix}_{t}, L_{j} , U_{j} \right)\right), \forall j \in I
\right\}. \ \ 
\end{align*} 
We remark that relaxation $\mathcal{F}_1^T$ is equivalent to the McCormick relaxation of the $tp$-Formulation in which  bilinear equations \eqref{eq:tp bilinear} are convexified via the McCormick envelopes in addition to the implied constraints $\sum_{t\in T_j}q^t_j = 1, \forall j \in I, \ L_j q_j^t \leq \sum_{i\in N_j^-}x_{ij}^t \leq U_j q_j^t, \forall j\in I, \forall t\in T_j$. The second   relaxation $\mathcal{F}_2^T$ is a strengthening of the first one by the addition of the column-wise extended formulation.

{Like in the case of the source-based formulation, we also considered a similar relaxation to $\mathcal{F}_1^T $, in which the column-wise (instead of the row-wise) extended formulation is used. However, also in this case, we empirically concluded that such a relaxation is weaker.}


{\bf Discretization relaxations}: Similarly to the relaxation of the source-based formulation in Section~\ref{sec:sBasedForm}, we define the following three relaxations based on the terminal formulation:
\begin{align*}
\cM_1^T(H):=
\left\{ (f,x)  \in \mathcal{F}_1^T  \,|\,  \ \begin{bmatrix} x_{ij}^t \end{bmatrix}_{(i,t)} \in  
\conv\left({\overline \cD}^{\textup{row}}_{(|N_j^-|, |T_j|, H)} \left ( \begin{bmatrix} l_{ij} \end{bmatrix}_{i} , \begin{bmatrix} u_{ij} \end{bmatrix}_{i} \right)\right), \forall j \in I
\right\},\\
\cM_2^T(H):=
\left\{ (f,x)  \in \mathcal{F}_2^T  \,|\,  \ \begin{bmatrix} x_{ij}^t \end{bmatrix}_{(i,t)} \in  
\conv\left({\overline \cD}^{\textup{row}}_{(|N_j^-|, |T_j|, H)} \left ( \begin{bmatrix} l_{ij} \end{bmatrix}_{i} , \begin{bmatrix} u_{ij} \end{bmatrix}_{i} \right)\right), \forall j \in I
\right\},\\
\cM_3^T(H):=
\left\{ (f,x)  \in \mathcal{F}_2^T  \,|\,  \ \begin{bmatrix} x_{ij}^t \end{bmatrix}_{(i,t)} \in  
\conv\left({\overline \cD}^{\textup{col}}_{(|N_j^-|, |T_j|, H)} \left ( \begin{bmatrix} l_{jt} \end{bmatrix}_{t} , \begin{bmatrix} u_{jt} \end{bmatrix}_{t} \right)\right), \forall j \in I
\right\}.
\end{align*}

{\bf Discretization restrictions}: Similarly to the restriction of the source-based formulation in Section~\ref{sec:sBasedForm}, we define the following two restrictions based on the terminal formulation:
\begin{align*}
\cG^{T}_1(H) :=
\left\{
(f,x)  \,|\,  \eqref{eq:arc capacity t}-\eqref{eq:spec t}, \ \begin{bmatrix} x_{ij}^t \end{bmatrix}_{(i,t)} \in  
 {\underline \cD}^{\textup{row}}_{(|N_j^-|, |T_j|,H)}  \left ( \begin{bmatrix} l_{ij} \end{bmatrix}_{i} , \begin{bmatrix} u_{ij} \end{bmatrix}_{i}, L_{j} , U_{j} \right), \forall j \in I
\right\},\\
\cG^{T}_2(H) :=
\left\{
(f,x)  \,|\,  \eqref{eq:arc capacity t}-\eqref{eq:spec t}, \ \begin{bmatrix} x_{ij}^t \end{bmatrix}_{(i,t)} \in  
 {\underline \cD}^{\textup{col}}_{(|N_j^-|, |T_j|,H)}  \left ( \begin{bmatrix} l_{jt} \end{bmatrix}_{t} , \begin{bmatrix} u_{jt} \end{bmatrix}_{t}, L_{j} , U_{j} \right), \forall j \in I
\right\}.
\end{align*}  

\subsubsection{Source and terminal-based rank formulation}
\label{sec:stBasedForm}

Let us define another set of decision variables $x_{ij}^{st}$ to denote the amount of flow on arc $(i,j)$ \textit{originated} at the source $s\in S_i$ and \textit{ended} at the terminal $t \in T_j$. In this formulation, first appeared in \cite{boland2016new}, we keep all the flow-related constraints of the source and terminal-based formulations and add the following extra flow constraints:
\begin{align}
&\sum_{j\in N_i^-: s \in S_j} x_{ji}^{st} = \sum_{j\in N_i^+: t \in T_j} x_{ij}^{st} & & \forall i\in  I, \forall s\in S_i, \forall t\in T_i \label{eq:flow conserve st} \\
&\sum_{s \in T_i}\sum_{t \in T_j} x_{ij}^{st} = f_{ij} & & \forall (i,j) \in A \label{eq:flow ij def st} \\
&\sum_{j\in N_i^+}\sum_{t\in T_j} x_{ij}^{st} = f_{si} & & \forall i \in I, \forall s \in S_i \label{eq:flow si def st} \\
&\sum_{i\in N_j^-}\sum_{s\in S_i} x_{ij}^{st} = f_{jt} & & \forall j \in I, \forall t \in T_j \label{eq:flow jt def st} \\
&  x_{ij}^{st} \ge 0 & & \forall (i,j) \in A, \forall s \in T_i , \forall t \in T_j .
\end{align}
The constraints corresponding to the specification requirements in this case are as following:
\begin{align}\label{eq:spec st}
\underline{\mu}_{k}^t  \sum_{j\in N_t^-} f_{jt}  \le
  \sum_{(s,j)\in A: s\in S,t\in T_j } \lambda_k^s x_{sj}^{st}   \le
\overline{\mu}_{k}^t  \sum_{j\in N_t^-} f_{jt}   \qquad \forall t\in T, \forall k\in K. 
\end{align}
Finally, to have a correct formulation we include the following bilinear equations to the model:
\begin{align}
&x_{ij}^{st} = q_{ij}^{st} f_{ij} && \qquad \forall (i,j) \in A, \forall s \in S_i, \forall t  \in T_j  \label{eq: stp bilinear}.
\end{align}
The next bilinear equations are implied, but their  McCormick envelopes can strengthen the relaxation,
\begin{align}
&q_{ij}^{st} = q_{i}^{s} q_{j}^{t} && \qquad \forall (i,j) \in A, \forall s \in S_i, \forall t  \in T_j. \label{eq: st bilinear q}
\end{align}
We will again rewrite the bilinear constraints in \eqref{eq: stp bilinear} as a set of rank restrictions on a matrix consisting of the decomposed flow variables  $ x_{ij}^{st} $ as follows:
\begin{equation}\label{eq:rank st}
\rank \left( \begin{bmatrix} x_{ij}^{st} \end{bmatrix}_{(s,t)\in S_i \times T_j} \right ) = 1 \qquad  \forall (i,j)\in A \cap(I \times I).
\end{equation}

{\bf Polyhedral relaxations}: Similar to the relaxations of the source and terminal-based formulation in Section~\ref{sec:sBasedForm} and \ref{sec:tBasedForm}, we define the following two polyhedral relaxations based on the source-terminal formulation:
\begin{align*}
\mathcal{F}_1^{ST}&:= \left\{
(f,x)  \,|\,  \ \eqref{eq:flow conserve st}-\eqref{eq:spec st}, \
\begin{bmatrix} x_{ij}^{st} \end{bmatrix}_{(s,t)} \in  
\conv\left({\cU}^{\textup{row}+}_{(|S_i|, |T_j|)} \left ( \begin{bmatrix} l_{si} \end{bmatrix}_{s} , \begin{bmatrix} u_{si} \end{bmatrix}_{s}, l_{ij} , u_{ij} \right)\right), \forall (i,j) \in A
\right\} \nonumber\\
& \cap \mathcal{F}_1^{S} \cap \mathcal{F}_1^{T} \cap Mc\\
\mathcal{F}_2^{ST} &:=
\left\{
(f,x) \in\mathcal{F}_1^{ST}   \,|\,  \ \ \begin{bmatrix} x_{ij}^{st} \end{bmatrix}_{(s,t)} \in 
\conv\left({\cU}^{\textup{col}+}_{(|S_i|, |T_j|)}  \left ( \begin{bmatrix} l_{jt} \end{bmatrix}_{t} , \begin{bmatrix} u_{jt} \end{bmatrix}_{t}, l_{ij} , u_{ij} \right)\right), \forall (i,j) \in A
\right\},
\end{align*} 
where $Mc$ denotes the McCormick relaxation of (\ref{eq: st bilinear q}), and we use ``$\cap$''  with slight abuse of notation since the sets do not live in the same variable space.

In our preliminary computations, we found that relaxation $\mathcal{F}_2^{ST}$ is too big and numerically unstable. Thus, we omitted from further discussing this relaxation.

\subsection{{Comparison with previous work}}
\label{sec:comparison}

Some of the relaxations presented in the previous section coincide with the well known relaxations of the pooling problem  from the literature. Table~\ref{table-Known-New} provides a summary as to which of these relaxations are already known and which of them, to the best of our knowledge, are new. 
\begin{table}[H]
\centering
\caption{New relaxations vs. known relaxations for standard and generalized pooling.}
\label{table-Known-New}
\begin{tabular}{ccc}
Formulation          & Standard Pooling & Generalized Pooling \\
\hline
$\mathcal{F}_1^{S}$  & Known \cite{tawarmalani2002convexification}           & Known  \cite{alfaki2013multi}               \\
$\mathcal{F}_1^{T}$  & Known \cite{alfaki2013strong}            & Known \cite{boland2016new}              \\
$\mathcal{F}_1^{ST}$ & Known \cite{boland2016new}            & Known \cite{boland2016new}         \\
$\mathcal{F}_2^{S}$  & Known*    \cite{alfaki2013strong}         & New                 \\
$\mathcal{F}_2^{T}$  & Known*   \cite{alfaki2013strong}          & New                 
\end{tabular}
\end{table}
We recall that $\mathcal{F}_1^{S}$  and $\mathcal{F}_1^{T}$ respectively correspond to the classical $pq$- and $tp$-relaxations for standard pooling. In the generalized pooling setting, $\mathcal{F}_1^{T}$ and  $\mathcal{F}_1^{S}$ correspond to the McCormick relaxations of the \textit{MCF-J-PQ} and \textit{MCF-I-PQ} formulations in \cite{boland2016new}. 
Moreover, $\mathcal{F}_1^{ST}$ coincides with the McCormick relaxation of the \textit{MCF-(I$\times$J)-PQ} formulation, also introduced in \cite{boland2016new}.

To the best of our knowledge, the relaxations $\mathcal{F}_2^{S} $ and $\mathcal{F}_2^{T}$ are new, at least in the case of generalized pooling. In fact, since   we have  $\mathcal{F}_1^{S} \cap \mathcal{F}_1^{T} = \mathcal{F}_2^{S} = \mathcal{F}_2^{T}$ for standard pooling,  these two relaxations happen to be equivalent to the $stp$-relaxation introduced in  \cite{alfaki2013strong}, which is the \textit{intersection} of $pq$- and $tp$-relaxations (the asterisk  in Table  \ref{table-Known-New} signalizes this fact). However, for generalized pooling instances, this equivalence  no longer holds. Therefore, $\mathcal{F}_2^{S} $ and $\mathcal{F}_2^{T}$ are new formulations in the generalized pooling setting. Moreover, our computational experiments have demonstrated that  the intersection of $\mathcal{F}_2^{S}$ and $\mathcal{F}_2^{T}$  yields the best bounds on average as shown in Table~\ref{table:LPvsMILP}, which demonstrates the value of formulating the generalized pooling problems with the rank-1 perspective.

\section{Computational results} \label{sec:maincomp}

\subsection{Software and hardware}

All of our experiments were run on a Windows 10 machine with 64-bit operating
system, x64 based processor with 2.19 GHz, and 32 GB RAM. We used Gurobi 7.5.1 to solve all the linear and mixed-integer linear programs.

\subsection{Instances}

We experimented with three different sets of generalized pooling problem instances. The first set of instances is composed by real-world data from the mining industry~\cite{Natashia2015}, the second set is a collection of instances from \cite{alfaki2013multi}, and the third one was generated by following the recipe proposed in \cite{alfaki2013multi}.

\noindent\textbf{Mining}: We start with a short description of the business problem. Raw material (coal or iron ore, for example) with known quality are blended at points in time in the so-called supply points. Orders for blended product arrive at known points in time and must be satisfied \textit{exactly} by blending previously blended material from the supply points. Costumers specify the minimum quality level for the final blended product. If the raw material is coal, quality level can be determined by the amount of ash and sulfur, for instance. A penalty is incurred if the minimum quality level is violated. The goal is to determine how much material to use from each supply point to meet the demand of each demand while minimizing total penalty. See \cite{Natashia2015} for more details about the problem, including how it can be formulated as an instance of the generalized pooling problem.
We experiment with 24 instances including quarterly, half-yearly and annual planning horizons. Details regarding the size of each instance are displayed in Table~\ref{table_size_coal_inst}.

\noindent\textbf{Literature}: In this class, we first considered 40 instances defined in \cite{alfaki2013multi}. However,  the standard $pq$-relaxation closes the duality gap for 28 of these instances. Therefore, we only focus on instances L1, L2, L3, L4, L5, L6, L12, L13, L14, L15,  {C2} and D1  that have non-trivial duality gap, and  report the  results of the experiments related to these 12 instances. See Table~\ref{table_size_literature_inst} for further information about them.  Instances whose name start with ``L'', were constructed by the authors of  \cite{alfaki2013multi} by adding pool-to-pool arcs to standard instances from the literature while  instances C2 and D1 were randomly generated. The objective function for all 12 instances is to minimize the total cost associated with each arc.

\noindent\textbf{Random}: We generated this set of  instances in the same way the authors of \cite{alfaki2013multi} generated instances C2 and D1. The only difference is that we increased the number of nodes and changed the density of the network to obtain more challenging instances. In total, we constructed 24 instances as described in Table~\ref{table_size_random_inst}. The objective function is to minimize the total cost associated with each arc. {These instances are available at \url{https://sites.google.com/view/asteroide-santana/pooling-problem}.}

\subsection{Primal bounds}

Primal bounds were available from the literature for both Mining~\cite{Natashia2015} and Literature~\cite{alfaki2013multi} instances. We compute primal bounds for the Random instances using the MILP discretization restrictions $\cG^{S}_k(H)$ and $\cG^{T}_k(H)$, $k=1,2$. In our experiments, we choose the discretization level $H=3$ for all the instances--meaning that each discretized variable may assume $2^3$ different values uniformly distributed within its domain. We then use the best primal bound among these four bounds to compute duality gap. Since the solver can take long to close the MILP duality gap, we set time limit of $1800s$ for each instance. If the MILP is not solved within this time limit, the MILP primal bound is taken as primal bound for the corresponding instance.

Table~\ref{table-pb-random-inst} displays the primal bound, the MILP duality gap, and run time, for each discretization restriction. The best performing method is highlighted in bold. Notice that $\cG^{S}_2(H)$, which is not based on the standard $pq$-formulation, has the best average performance. This is true in terms of finding the best primal bound, closing the MILP duality gap as well as running time.

\begin{table}[H]
\centering
\caption{Primal bounds via discretization for Random instances. Here, ``Bound'', ``Gap'' and ``Time'' are the MILP primal bound, the percentage optimality gap of the MILP model upon termination, and run time in seconds, respectively.}
\label{table-pb-random-inst}
\resizebox{\textwidth}{!}{
\begin{tabular}{c|ccc|ccc|ccc|ccc|}
 \multicolumn{1}{c}{}& \multicolumn{3}{c}{$\cG^{S}_1(H)$} & \multicolumn{3}{c}{$\cG^{S}_2(H)$} & \multicolumn{3}{c}{$\cG^{T}_1(H)$} & \multicolumn{3}{c}{$\cG^{T}_2(H)$} \\
Inst & Bound       & Gap       & Time       & Bound       & Gap       & Time       & Bound       & Gap       & Time       & Bound       & Gap       & Time \\
\hline

        F1 &   -1933.62 &       0.00 &      22.72 & {\bf -1968.17} &       0.00 &       5.73 &   -1968.00 &       0.01 &       6.30 &   -1957.15 &       0.00 &      21.05 \\

        F2 &   -4223.78 &       5.63 &    1800.02 & {\bf -4340.93} &       1.52 &    1800.03 &   -4265.44 &       3.69 &    1800.03 &   -4287.22 &       3.56 &    1800.03 \\

        F3 &   -1293.56 &       2.96 &    1800.08 &   -1344.41 &       0.00 &     620.59 & {\bf -1344.43} &       3.34 &    1800.06 &   -1251.69 &      10.29 &    1800.08 \\

        F4 &       0.00 &       0.00 &       1.30 &       0.00 &       0.00 &       3.27 &       0.00 &       0.00 &       2.20 &       0.00 &       0.00 &       3.11 \\

        F5 &   -2735.68 &      12.86 &    1800.02 & {\bf -2892.68} &       4.41 &    1800.05 &   -2721.91 &      13.40 &    1800.03 &   -2685.98 &      14.75 &    1800.03 \\

        F6 & {\bf -4818.56} &       8.05 &    1800.09 &   -4770.56 &       8.40 &    1800.03 &   -4768.38 &       9.64 &    1800.05 &   -4785.93 &       8.67 &    1800.19 \\

        F7 &   -2214.34 &      35.81 &    1800.03 & {\bf -2431.58} &      24.13 &    1800.03 &   -2322.47 &      29.88 &    1800.05 &   -2290.95 &      29.21 &    1800.06 \\

        F8 &   -2493.85 &      27.16 &    1800.17 & {\bf -2895.15} &      10.06 &    1800.05 &   -2893.45 &       4.73 &    1800.03 &   -2733.37 &      14.15 &    1800.06 \\

        F9 &   -2290.91 &      17.91 &    1800.03 & {\bf -2552.66} &       1.97 &    1800.05 &   -2536.85 &       4.54 &    1800.06 &   -2192.65 &      21.55 &    1800.06 \\

       F10 &   -2721.10 &       5.70 &    1800.03 & {\bf -3006.92} &       0.00 &    1292.48 &   -2969.85 &       0.77 &    1800.05 &   -2897.09 &       0.70 &    1800.03 \\

       F11 &   -2104.35 &       0.00 &      26.83 & {\bf -2159.43} &       0.00 &       6.55 &   -2159.43 &       0.00 &       7.03 &   -2143.44 &       0.00 &       8.50 \\

       F12 &   -4541.65 &       2.05 &    1800.02 & {\bf -4594.01} &       0.01 &    1107.47 &   -4583.04 &       1.11 &    1800.06 &   -4546.62 &       2.03 &    1800.05 \\

       F13 &   -2786.81 &       6.54 &    1800.03 & {\bf -2872.81} &       0.01 &    1320.28 &   -2867.10 &       2.22 &    1800.03 &   -2812.18 &       6.09 &    1800.05 \\

       F14 &    -492.11 &       0.00 &      11.27 & {\bf -497.23} &       0.00 &       0.91 &    -497.23 &       0.00 &       0.95 &    -495.81 &       0.00 &       4.14 \\

       F15 &   -2365.46 &       8.43 &    1800.02 & {\bf -2414.56} &      10.43 &    1800.03 &   -2396.89 &      11.97 &    1800.06 &   -2379.11 &      13.12 &    1800.02 \\

       F16 &   -5824.86 &       0.01 &     271.45 &   -5834.27 &       0.00 &     274.00 & {\bf -5838.04} &       0.00 &      59.13 &   -5831.01 &       0.01 &     326.98 \\

       F17 &   -2100.98 &       0.00 &      17.16 &   -2098.36 &       0.00 &       4.41 &   -2098.36 &       0.00 &       8.25 & {\bf -2105.35} &       0.00 &       9.37 \\

       F18 &    -744.71 &       0.01 &     194.59 & {\bf -760.57} &       0.00 &      22.61 &    -760.55 &       0.00 &      14.98 &    -745.66 &       0.01 &     130.30 \\

       F19 &    -555.37 &       0.00 &       1.03 & {\bf -582.87} &       0.00 &       0.53 & {\bf -582.87} &       0.01 &       0.36 &    -555.37 &       0.00 &       0.66 \\

       F20 &   -3419.62 &       5.44 &    1800.02 & {\bf -3599.41} &       0.00 &     115.48 &   -3548.90 &       0.00 &     318.98 &   -3537.36 &       0.01 &     842.69 \\

       F21 &   -8094.23 &      30.97 &    1800.08 & {\bf -8417.90} &      25.51 &    1800.09 &   -7567.30 &      40.54 &    1800.14 &   -7312.21 &      44.58 &    1800.05 \\

       F22 &   -3629.75 &      14.18 &    1800.00 & {\bf -3893.33} &       9.51 &    1800.03 &   -3854.77 &      10.53 &    1800.03 &   -3828.98 &       6.23 &    1800.03 \\

       F23 &   -6611.50 &      45.11 &    1800.06 & {\bf -6667.33} &      43.43 &    1800.08 &   -6026.67 &      61.57 &    1800.09 &   -6458.53 &      49.26 &    1800.08 \\

       F24 &    -999.81 &       0.00 &     140.50 & {\bf -1025.35} &       0.00 &      26.97 &   -1009.19 &       0.00 &      51.94 &   -1005.53 &       0.00 &     186.87 \\

\hline
Ave. & -2874.86 & 9.53 & 1153.65 & \textbf{-2984.19} & 5.81 & 950.07 & -2899.21 & 8.25 & 1069.62 & -2868.30 & 9.34 & 1113.94
\end{tabular}}
\end{table}

\subsection{Dual bounds via LP relaxations}

%
%
%
%
%
%
%
%

We compute dual bounds  using three types of LP relaxations which we will refer to as ``Light", ``Medium" and ``Heavy". Light LP relaxations are given by $\mathcal{F}_k^S, \mathcal{F}_k^T$, $k=1,2$ while Medium LP relaxations are defined as $\mathcal{F}_k^S \cap \mathcal{F}_l^T$, $k,l=1,2$. The only Heavy LP relaxation considered is $\mathcal{F}_1^{ST}$ from  \cite{boland2016new}.
Duality gap and run time of each of these nine methods are reported for all the three sets of instances in Tables~\ref{table-gap-mining-inst}, \ref{table-gap-literature-inst} and \ref{table-gap-random-inst}. The best performing method is highlighted in bold for each instance.

By construction, the method based on $\mathcal{F}_2^S \cap \mathcal{F}_2^T$  gives the strongest dual bound among Light and Medium LP relaxations. It is interesting to observe that it also performs better than the Heavy LP relaxation $\mathcal{F}_{1}^{ST}$
in terms of both average duality gap and run time. In fact, there are only three  instances in which $\mathcal{F}_{1}^{ST}$ is strictly better than  all of the Light and Medium LP relaxations (these are instances  2009Q4, 2011H1 and 2011H2 from the Mining set). On the other hand, the Heavy LP relaxation $\mathcal{F}_{1}^{ST}$ takes significantly longer than lighter LP relaxations and runs into memory issues occasionally. For these instances, we report the best bound (and respective time) among all the other methods (see numbers in parentheses in Tables~\ref{table-gap-mining-inst} and \ref{table-gap-random-inst}), and use those figures in the averages. 

Finally, we point out the relative success of the Light LP relaxations, which typically take much shorter amount of time than their Medium LP counterparts with similar optimality gaps proven. This has motivated us to only focus on the MILP relaxations of the Light LP relaxations as discussed in the next section. We also note that a comparison of the best performing relaxations from the Light and Medium LP relaxations  in terms of their average performances is  provided later in Table~\ref{table:LPvsMILP}.


\begin{table}[!h]
\centering
\caption{Duality gaps for Mining instances via LP. Here, ``Paper'' is the relaxation used in \cite{Natashia2015}, the paper that introduced these instances in the literature.}
\label{table-gap-mining-inst}
\resizebox{\textwidth}{!}{
\begin{tabular}{c|cc|cc|cc|cc|cc|cc|cc|cc|cc|cc|}
 \multicolumn{1}{c}{}   &  \multicolumn{2}{c}{Paper~\cite{Natashia2015}}  & \multicolumn{2}{c}{$\mathcal{F}_1^S$} & \multicolumn{2}{c}{$\mathcal{F}_2^S$} & \multicolumn{2}{c}{$\mathcal{F}_1^T$} & \multicolumn{2}{c}{$\mathcal{F}_2^T$} & \multicolumn{2}{c}{$\mathcal{F}_1^S\cap\mathcal{F}_1^T$} & \multicolumn{2}{c}{$\mathcal{F}_2^S\cap\mathcal{F}_1^T$} & \multicolumn{2}{c}{$\mathcal{F}_1^S\cap\mathcal{F}_2^T$} & \multicolumn{2}{c}{$\mathcal{F}_2^S\cap\mathcal{F}_2^T$} & \multicolumn{2}{c}{$\mathcal{F}_{1}^{ST}$} \\
Inst & Gap & Time  & Gap              & Time             & Gap              & Time             & Gap              & Time             & Gap              & Time             & Gap                            & Time                            & Gap                            & Time                            & Gap                            & Time                            & Gap                            & Time                            & Gap                & Time               \\
\hline
2009H2    & 37.00 & 0.05  & 7.13              & 0.53 & 4.06              & 1.12  & 4.58              & 1.23  & 4.58              & 9.28  & 4.27                                           & 5.13  & \textbf{3.59}                                  & 5.89   & 4.27                                           & 7.16   & \textbf{3.59}                                  & 6.64   & 4.02                   & 1697.27  \\
2009Q3    & 29.36 & 0.03  & 3.84              & 0.06 & 1.59              & 0.23  & 2.16              & 0.08  & 2.16              & 0.17  & 2.01                                           & 0.34    & \textbf{1.45}                                  & 0.69   & 2.01                                           & 0.61   & \textbf{1.45}                                  & 0.47  & 1.96                   & 1.23    \\
2009Q4    & 41.86 & 0.02 & 24.93             & 0.13 & 14.49             & 0.57  & 17.11             & 0.30  & 17.11             & 0.78  & 16.94                                          & 1.20      & 13.77                                          & 1.75   & 16.94                                          & 1.36   & 13.77                                          & 1.98  & \textbf{10.84}         & 22.59 \\
2010Y    & 26.01 & 0.25 & 10.37             & 5.16 & 8.58              & 10.99 & 8.91              & 33.62 & 8.91              & 61.63 & 8.87                                           & 101.08 & \textbf{8.14}                                  & 86.91  & 8.87                                           & 141.66 & \textbf{8.14}                                  & 92.30   & (\textbf{8.14}) & 86.91   \\
2010H1    & 32.42  & 0.08 & 14.64             & 1.02 & 13.42             & 1.95  & 13.38             & 8.97  & 13.38             & 6.23  & 13.32                                          & 10.09   & \textbf{12.77}                                 & 17.47  & 13.32                                          & 18.48  & \textbf{12.77}                                 & 15.73  & 12.78                  & 285.80 \\
2010H2    & 14.83  & 0.09 & 3.94              & 0.81 & 1.84              & 2.09  & 2.45              & 3.27  & 2.45              & 5.34  & 2.45                                           & 10.60  & \textbf{1.80}                                  & 16.55  & 2.45                                           & 15.61  & \textbf{1.80}                                  & 14.89   & 2.26                   & 480.52 \\
2010Q1    & 19.74 & 0.03 & 4.35              & 0.20 & 3.90              & 0.75  & 4.31              & 0.55  & 4.31              & 0.66  & 4.14                                           & 2.03    & \textbf{3.88}                                  & 1.81   & 4.14                                           & 2.50   & \textbf{3.88}                                  & 2.36  & 3.94                   & 13.33  \\
2010Q2    & 35.38  & 0.03 & 21.41             & 0.11 & 18.59             & 0.61  & 18.58             & 0.27  & 18.58             & 0.61  & 18.55                                          & 1.28    & \textbf{17.90}                                 & 1.39   & 18.55                                          & 1.38   & \textbf{17.90}                                 & 0.98  & 18.26                  & 5.66    \\
2010Q3   & 20.33 & 0.02 & 3.84              & 0.09 & 1.54              & 4.05  & 2.56              & 0.13  & 2.56              & 0.39  & 2.56                                           & 0.61    & \textbf{1.49}                                  & 0.94   & 2.56                                           & 1.11   & \textbf{1.49}                                  & 0.98   & 2.23                   & 4.20   \\
2010Q4   & 28.62 & 0.05 & 11.43             & 0.12 & 9.95              & 0.36  & 9.43              & 0.58  & 9.43              & 0.94  & 9.14                                           & 1.00    & \textbf{8.83}                                  & 1.37   & 9.14                                           & 1.48   & \textbf{8.83}                                  & 1.56  & 8.75                   & 47.97  \\
2011Y   & 19.30 & 0.14 & 2.45              & 2.83 & 1.50              & 7.64  & 1.43              & 7.95  & 1.43              & 16.39 & 1.41                                           & 57.67   & \textbf{1.25}                                  & 322.28 & 1.41                                           & 60.89  & \textbf{1.25}                                  & 88.55 & 1.28                   & 1161.44 \\
2011H1   & 9.07 & 0.06  & 2.30              & 0.48 & 1.40              & 0.91  & 1.24              & 0.83  & 1.24              & 1.91  & 1.23                                           & 2.36   & 1.13                                           & 1.95   & 1.23                                           & 2.77   & 1.13                                           & 5.39  & \textbf{1.10}          & 14.53   \\
2011H2   & 22.21 & 0.06 & 2.17              & 0.77 & 1.43              & 1.56  & 1.37              & 1.53  & 1.37              & 2.38  & 1.34                                           & 7.75   & 1.26                                           & 10.12  & 1.34                                           & 12.62  & 1.26                                           & 21.31 & \textbf{1.24}          & 172.06  \\
2011Q1   & 10.51 & 0.02 & 1.78              & 0.04 & 1.03              & 0.15  & 0.83              & 0.11  & 0.83              & 0.34  & 0.82                                           & 0.52    & \textbf{0.72}                                  & 0.53   & 0.82                                           & 0.47   & \textbf{0.72}                                  & 0.77  & \textbf{0.72}          & 1.16   \\
2011Q2   & 4.04  & 0.02 & 3.19              & 0.05 & 2.29              & 0.16  & 2.53              & 0.12  & 2.53              & 0.40  & 2.53                                           & 0.34   & \textbf{2.10}                                  & 0.45   & 2.53                                           & 0.63   & \textbf{2.10}                                  & 0.98   & 2.24                   & 3.75   \\
2011Q3   & 10.04 & 0.00 & 0.39              & 0.02 & 0.07              & 0.03  & 0.15              & 0.03  & 0.15              & 0.06  & 0.15                                           & 0.14   & \textbf{0.07}                                  & 0.08   & 0.15                                           & 0.11   & \textbf{0.07}                                  & 0.14   & 0.15                   & 0.25   \\
2011Q4   & 16.09 & 0.02 & 1.01              & 0.06 & 0.57              & 0.10  & 0.99              & 0.09  & 0.99              & 0.30  & 0.90                                           & 0.53   & \textbf{0.56}                                  & 0.73   & 0.90                                           & 0.63   & \textbf{0.56}                                  & 0.97   & 0.83                   & 3.08   \\
2012Y   & 8.20 & 0.11  & 2.79              & 1.78 & 1.72              & 5.61  & 2.37              & 7.56  & 2.37              & 16.38 & 2.19                                           & 21.64 & \textbf{1.57}                                  & 30.39  & 2.19                                           & 29.50  & \textbf{1.57}                                  & 41.34   & (\textbf{1.57})                   & 30.39    \\
2012H1   & 4.99  & 0.08 & 3.35              & 1.27 & 2.20              & 2.50  & 3.00              & 2.34  & 3.00              & 5.17  & 2.77                                           & 10.77  & \textbf{2.02}                                  & 16.66  & 2.77                                           & 13.19  & \textbf{2.02}                                  & 28.83  & (\textbf{2.02})                   & 16.66   \\
2012H2   & 10.21 & 0.02 & 0.97              & 0.11 & 0.42              & 0.47  & 0.74              & 0.19  & 0.74              & 0.33  & 0.66                                           & 1.02   & \textbf{0.42}                                  & 0.72   & 0.66                                           & 0.91   & \textbf{0.42}                                  & 1.33  & 0.64                   & 4.75    \\
2012Q1   & 13.93 & 0.01 & 9.31              & 0.03 & 4.43              & 0.08  & 8.66              & 0.09  & 8.66              & 0.22  & 8.25                                           & 0.41    & \textbf{4.16}                                  & 0.33   & 8.25                                           & 0.31   & \textbf{4.16}                                  & 0.75  & 8.13                   & 2.37   \\
2012Q2   & 1.68  & 0.02 & 0.70              & 0.03 & 0.30              & 0.09  & 0.49              & 0.08  & 0.49              & 0.23  & 0.49                                           & 0.34    & \textbf{0.21}                                  & 0.28   & 0.49                                           & 0.53   & \textbf{0.21}                                  & 0.41  & 0.49                   & 0.95   \\
2012Q3   & 5.92  & 0.02 & 1.13              & 0.05 & 0.56              & 0.13  & 0.99              & 0.06  & 0.99              & 0.20  & 0.92                                           & 0.38    & \textbf{0.56}                                  & 0.33   & 0.92                                           & 0.42   & \textbf{0.56}                                  & 0.86  & 0.89                   & 1.52   \\
2012Q4   & 25.81 & 0.00 & 4.91              & 0.02 & 3.50              & 0.02  & 1.91              & 0.03  & 1.91              & 0.03  & 1.91                                           & 0.05     & \textbf{1.91}                                  & 0.06   & 1.91                                           & 0.08   & \textbf{1.91}                                  & 0.09  & 1.91                   & 0.22  \\
\hline
Ave. & 18.65 & 0.05	& 5.93	& 0.66	& 4.14	& 1.76	& 4.59	& 2.92	& 4.59	& 5.43	& 4.49	& 9.89	& \textbf{3.82}	& 21.65	& 4.49	& 13.10 &	\textbf{3.82}	& 13.73	& 4.02	&169.11

\end{tabular}}
\end{table}

\begin{table}[!h]
\centering
\caption{Duality gaps for Literature instances via LP.}
\label{table-gap-literature-inst}
\resizebox{\textwidth}{!}{
\begin{tabular}{c|cc|cc|cc|cc|cc|cc|cc|cc|cc|}
 \multicolumn{1}{c}{}    & \multicolumn{2}{c}{$\mathcal{F}_1^S$} & \multicolumn{2}{c}{$\mathcal{F}_2^S$} & \multicolumn{2}{c}{$\mathcal{F}_1^T$} & \multicolumn{2}{c}{$\mathcal{F}_2^T$} & \multicolumn{2}{c}{$\mathcal{F}_1^S\cap\mathcal{F}_1^T$} & \multicolumn{2}{c}{$\mathcal{F}_2^S\cap\mathcal{F}_1^T$} & \multicolumn{2}{c}{$\mathcal{F}_1^S\cap\mathcal{F}_2^T$} & \multicolumn{2}{c}{$\mathcal{F}_2^S\cap\mathcal{F}_2^T$} & \multicolumn{2}{c}{$\mathcal{F}_{1}^{ST}$} \\
Inst & Gap              & Time             & Gap              & Time             & Gap              & Time             & Gap              & Time             & Gap                            & Time                            & Gap                            & Time                            & Gap                            & Time                            & Gap                            & Time                            & Gap                & Time               \\
\hline
L1             & 1.01              & 0.00          & \textbf{0.86}     & 0.00          & 1.01              & 0.00          & 1.01              & 0.00          & 1.01                                           & 0.02                 & \textbf{0.86}                                  & 0.00          & 1.01                                           & 0.00          & \textbf{0.86}                                  & 0.00         & 1.01                   & 0.00    \\
L2             & 55.24             & 0.00          & 55.24             & 0.02          & 55.74             & 0.00          & \textbf{52.83}    & 0.00          & 55.24                                          & 0.00               & 55.24                                          & 0.02          & \textbf{52.83}                                 & 0.02          & \textbf{52.83}                                 & 0.00      & 55.24                  & 0.02         \\
L3             & \textbf{4.55}     & 0.00          & \textbf{4.55}     & 0.00          & \textbf{4.55}     & 0.00          & \textbf{4.55}     & 0.00          & \textbf{4.55}                                  & 0.00                 & \textbf{4.55}                                  & 0.02          & \textbf{4.55}                                  & 0.00          & \textbf{4.55}                                  & 0.00       & \textbf{4.55}          & 0.03      \\
L4            & \textbf{2.45}     & 0.02          & \textbf{2.45}     & 0.06          & \textbf{2.45}     & 0.00          & \textbf{2.45}     & 0.00          & \textbf{2.45}                                  & 0.02                 & \textbf{2.45}                                  & 0.05          & \textbf{2.45}                                  & 0.03          & \textbf{2.45}                                  & 0.09       & \textbf{2.45}          & 0.16      \\
L5            & 10.80             & 0.02          & 10.80             & 0.00          & 11.26             & 0.00          & \textbf{9.60}     & 0.02          & 10.80                                          & 0.02                  & 10.80                                          & 0.02          & \textbf{9.60}                                  & 0.02          & \textbf{9.60}                                  & 0.03       & 10.80                  & 0.03     \\
L6            & 22.22             & 0.00          & 22.06             & 0.00          & \textbf{20.37}    & 0.00          & \textbf{20.37}    & 0.00          & \textbf{20.37}                                 & 0.00          & \textbf{20.37}         & 0.00          & \textbf{20.37}                                 & 0.00          & \textbf{20.37}                                 & 0.00          & \textbf{20.37}                                 & 0.00          \\
L12           & \textbf{25.00}    & 0.02          & \textbf{25.00}    & 0.00          & \textbf{25.00}    & 0.00          & \textbf{25.00}    & 0.00          & \textbf{25.00}                                 & 0.00          & \textbf{25.00}         & 0.00          & \textbf{25.00}                                 & 0.00          & \textbf{25.00}                                 & 0.00          & \textbf{25.00}                                 & 0.00          \\
L13           & \textbf{66.67}    & 0.00          & \textbf{66.67}    & 0.00          & \textbf{66.67}    & 0.00          & \textbf{66.67}    & 0.00          & \textbf{66.67}                                 & 0.00          & \textbf{66.67}         & 0.00          & \textbf{66.67}                                 & 0.00          & \textbf{66.67}                                 & 0.00          & \textbf{66.67}                                 & 0.00          \\
L14           & 16.67             & 0.00          & 16.67             & 0.00          & 16.67             & 0.00          & \textbf{6.67}     & 0.00          & 16.67                                          & 0.00                   & 16.67                                          & 0.00          & \textbf{6.67}                                  & 0.00          & \textbf{6.67}                                  & 0.00         & 16.67                  & 0.00  \\
L15            & 37.41             & 0.00          & \textbf{25.88}    & 0.00          & \textbf{25.88}    & 0.00          & \textbf{25.88}    & 0.00          & \textbf{25.88}                                 & 0.02               & \textbf{25.88}                                 & 0.00          & \textbf{25.88}                                 & 0.00          & \textbf{25.88}                                 & 0.02         & \textbf{25.88}         & 0.00      \\
C2             & \textbf{1.23}     & 0.03          & \textbf{1.23}     & 0.02          & \textbf{1.23}     & 0.03          & \textbf{1.23}     & 0.03          & \textbf{1.23}                                  & 0.08                & \textbf{1.23}                                  & 0.08          & \textbf{1.23}                                  & 0.09          & \textbf{1.23}                                  & 0.06          & \textbf{1.23}          & 0.33    \\
D1             & \textbf{1.06}     & 0.09          & \textbf{1.06}     & 0.25          & \textbf{1.06}     & 0.12          & \textbf{1.06}     & 0.22          & \textbf{1.06}                                  & 0.55                 & \textbf{1.06}                                  & 0.94          & \textbf{1.06}                                  & 0.94          & \textbf{1.06}                                  & 1.17        & \textbf{1.06}          & 3.92     \\
\hline
Ave.  & 20.36    & 0.01 & 19.37    & 0.03 & 19.32    & 0.01 & 18.11    & 0.02 & 19.24                                & 0.06  & 19.23                                & 0.09 & 18.11                                 & 0.09 & \textbf{18.10}                             & 0.11 & 19.24         & 0.37
\end{tabular}}
\end{table}

\begin{table}[!h]
\centering
\caption{Duality gaps for Random instances via LP.}
\label{table-gap-random-inst}
\resizebox{\textwidth}{!}{
\begin{tabular}{c|cc|cc|cc|cc|cc|cc|cc|cc|cc|}
 \multicolumn{1}{c}{}   & \multicolumn{2}{c}{$\mathcal{F}_1^S$} & \multicolumn{2}{c}{$\mathcal{F}_2^S$} & \multicolumn{2}{c}{$\mathcal{F}_1^T$} & \multicolumn{2}{c}{$\mathcal{F}_2^T$} & \multicolumn{2}{c}{$\mathcal{F}_1^S\cap\mathcal{F}_1^T$} & \multicolumn{2}{c}{$\mathcal{F}_2^S\cap\mathcal{F}_1^T$} & \multicolumn{2}{c}{$\mathcal{F}_1^S\cap\mathcal{F}_2^T$} & \multicolumn{2}{c}{$\mathcal{F}_2^S\cap\mathcal{F}_2^T$} & \multicolumn{2}{c}{$\mathcal{F}_{1}^{ST}$} \\
Inst & Gap              & Time             & Gap              & Time             & Gap              & Time             & Gap              & Time             & Gap                            & Time                            & Gap                            & Time                            & Gap                            & Time                            & Gap                            & Time                            & Gap                & Time               \\
\hline
F1            & \textbf{3.79}     & 0.02          & \textbf{3.79}     & 0.02          & 5.02              & 0.03          & 4.09              & 0.06          & \textbf{3.79}                                  & 0.05                 & \textbf{3.79}                                  & 0.12           & \textbf{3.79}                                  & 0.09           & \textbf{3.79}                                  & 0.09             & \textbf{3.79}          & 0.22  \\
F2            & 4.81              & 0.39          & 4.60              & 1.02          & 5.15              & 0.41          & 4.50              & 0.89          & 4.30                                           & 1.27                   & 4.29                                           & 3.43           & \textbf{4.28}                                  & 3.06           & \textbf{4.28}                                  & 5.98         & 4.29                   & 13.70   \\
F3            & 13.52             & 0.27          & \textbf{13.19}    & 1.27          & 14.57             & 0.70          & 14.57             & 1.41          & \textbf{13.19}                                 & 1.55                 & \textbf{13.19}                                 & 2.82           & \textbf{13.19}                                 & 2.00           & \textbf{13.19}                                 & 6.44           & 13.19                  & 22.39    \\
F4            & \textbf{0.00}     & 0.14          & \textbf{0.00}     & 2.48          & \textbf{0.00}     & 1.03          & \textbf{0.00}     & 1.72          & \textbf{0.00}                                  & 1.31                   & \textbf{0.00}                                  & 4.80           & \textbf{0.00}                                  & 3.37           & \textbf{0.00}                                  & 6.25           & \textbf{0.00}          & 6.61  \\
F5            & 14.58             & 0.42          & 11.02             & 1.81          & 13.06             & 0.34          & 11.96             & 1.06          & 11.71                                          & 1.19                 & \textbf{10.97}                                 & 3.59           & 11.66                                          & 3.03           & \textbf{10.97}                                 & 6.19           & 11.65                  & 22.69   \\
F6            & 10.69             & 0.31          & 10.03             & 1.73          & 12.25             & 0.27          & 10.45             & 1.17          & 10.15                                          & 3.19                  & 9.92                                           & 8.11           & 9.97                                           & 6.56           & \textbf{9.91}                                  & 12.42        & 9.94                   & 60.92    \\
F7            & 33.72             & 1.05          & 33.09             & 3.25          & 34.91             & 1.20          & 34.04             & 5.08          & 33.09                                          & 5.31                 & 33.09                                          & 12.50          & \textbf{33.05}                                 & 12.45          & \textbf{33.05}                                 & 22.45           & 33.09                  & 86.13  \\
F8            & 17.15             & 0.58          & \textbf{16.90}    & 1.77          & 17.46             & 0.53          & 17.06             & 1.81          & \textbf{16.90}                                 & 4.22                & \textbf{16.90}                                 & 8.47           & \textbf{16.90}                                 & 7.89           & \textbf{16.90}                                 & 14.05           & \textbf{16.90}         & 22.58   \\
F9            & 18.13             & 0.50          & 17.17             & 1.03          & 20.75             & 0.38          & 17.78             & 1.15          & 17.24                                          & 3.48              & \textbf{17.15}                                 & 8.91           & 17.16                                          & 5.77           & \textbf{17.15}                                 & 16.30        & 17.16                  & 19.64        \\
F10           & 7.84              & 0.31          & \textbf{7.81}     & 1.30          & 8.08              & 0.61          & 7.86              & 1.41          & \textbf{7.81}                                  & 2.66                 & \textbf{7.81}                                  & 7.27           & \textbf{7.81}                                  & 4.25           & \textbf{7.81}                                  & 9.17            & \textbf{7.81}          & 12.98  \\
F11           & 5.43              & 0.02          & 5.32              & 0.05          & 5.71              & 0.03          & 5.20              & 0.08          & 5.28                                           & 0.06                  & 5.28                                           & 0.16           & \textbf{5.19}                                  & 0.16           & \textbf{5.19}                                  & 0.23        & 5.28                   & 0.39      \\
F12           & 4.57              & 0.14          & 4.54              & 0.30          & 5.14              & 0.19          & 4.78              & 0.34          & 4.48                                           & 1.06                 & 4.48                                           & 1.69           & \textbf{4.47}                                  & 1.64           & \textbf{4.47}                                  & 2.31       & 4.48                   & 8.73        \\
F13           & 7.20              & 0.67          & 7.15              & 2.11          & 8.12              & 0.69          & 7.92              & 1.81          & 7.17                                           & 1.86               & 7.14                                           & 4.45           & 7.16                                           & 3.84           & \textbf{7.13}                                  & 7.48       & 7.16                   & 32.11         \\
F14           & \textbf{0.00}     & 0.08          & \textbf{0.00}     & 0.48          & \textbf{0.00}     & 0.30          & \textbf{0.00}     & 0.80          & \textbf{0.00}                                  & 0.95                 & \textbf{0.00}                                  & 2.00           & \textbf{0.00}                                  & 2.02           & \textbf{0.00}                                  & 1.75             & \textbf{0.00}          & 3.83  \\
F15           & 17.95             & 0.23          & 17.88             & 0.78          & 18.82             & 0.39          & 18.33             & 1.61          & 16.36                                          & 1.84               & \textbf{16.33}                                 & 2.77           & 16.36                                          & 3.03           & \textbf{16.33}                                 & 6.06        & 16.36                  & 14.53        \\
F16           & 4.64              & 0.06          & 4.29              & 0.09          & 4.41              & 0.05          & 4.31              & 0.11          & 4.28                                           & 0.28                  & \textbf{4.26}                                  & 0.42           & 4.28                                           & 0.38           & \textbf{4.26}                                  & 0.64       & 4.27                   & 2.31       \\
F17           & 10.24             & 0.02          & 10.24             & 0.05          & 9.63              & 0.06          & \textbf{9.62}     & 0.11          & \textbf{9.62}                                  & 0.12                & \textbf{9.62}                                  & 0.16           & \textbf{9.62}                                  & 0.17           & \textbf{9.62}                                  & 0.27            & \textbf{9.62}          & 0.64    \\
F18           & \textbf{16.63}    & 0.09          & \textbf{16.63}    & 0.34          & \textbf{16.63}    & 0.19          & \textbf{16.63}    & 0.56          & \textbf{16.63}                                 & 0.80              & \textbf{16.63}                                 & 0.73           & \textbf{16.63}                                 & 1.08           & \textbf{16.63}                                 & 2.53             & \textbf{16.63}         & 1.87     \\
F19           & \textbf{1.06}     & 0.02          & \textbf{1.06}     & 0.02          & 2.73              & 0.03          & \textbf{1.06}     & 0.03          & \textbf{1.06}                                  & 0.03                  & \textbf{1.06}                                  & 0.03           & \textbf{1.06}                                  & 0.05           & \textbf{1.06}                                  & 0.09             & \textbf{1.06}          & 0.08 \\
F20           & 12.74             & 0.11          & 12.11             & 0.16          & 13.38             & 0.14          & 12.28             & 0.31          & 11.52                                          & 0.91                 & 11.48                                          & 0.88           & 11.50                                          & 1.00           & \textbf{11.46}                                 & 1.34         & 11.52                  & 2.63      \\
F21           & 26.02             & 12.38         & 25.57             & 55.88         & 26.49             & 8.72          & 25.71             & 50.89         & 25.42                                          & 70.72               & 25.40                                          & 165.03         & 25.37                                          & 133.67         & \textbf{25.36}                                 & 310.98           & (\textbf{25.36})                  & 310.98 \\
F22           & 22.64             & 0.23          & 22.17             & 0.55          & 23.43             & 0.34          & 22.61             & 0.66          & 21.91                                          & 1.42                 & \textbf{21.91}                                 & 3.19           & \textbf{21.91}                                 & 2.86           & \textbf{21.91}                                 & 4.34            & \textbf{21.91}         & 10.83  \\
F23           & 43.92             & 10.64         & 43.58             & 49.09         & 47.04             & 12.15         & 45.05             & 49.55         & 43.39                                          & 53.61               & 43.35                                          & 493.14         & 43.16                                          & 224.12         & \textbf{43.13}                                 & 207.30          & (\textbf{43.13})                  & 207.30 \\
F24           & 11.19             & 0.22          & 11.19             & 0.50          & 11.19             & 1.02          & \textbf{11.18}    & 2.23          & 11.19                                          & 2.23                 & 11.19                                          & 5.81           & \textbf{11.18}                                 & 7.45           & \textbf{11.18}                                 & 11.38          & 11.19                  & 5.44    \\
\hline
Ave. & 12.85    & 1.20 & 12.47    & 5.25 & 13.50    & 1.24 & 12.79    & 5.20 & 12.35                                 & 6.67  & 12.30                                 & 30.85 & 12.32                                 & 17.91 & \textbf{12.28}                                 & 27.34 & 12.32         & 36.23
\end{tabular}}
\end{table}

\subsection{Dual bounds via MILP relaxations}

Besides using the LP relaxations, we also compute dual bounds for all the instances using the MILP discretization relaxations 
$\cM^{S}_k(H)$ and $\cM^{T}_k(H)$, $k=1,2,3$. {Recall that $\cM^{S}_1(H)$ and $\cM^{S}_2(H)$ are obtained from discretizing the ratio variables $q^s_j$ in (\ref{eq:pq bilinear}), whereas $\cM^{S}_3(H)$ is obtained from discretizing the newly introduced ratio variables $q_{ij}$ in (\ref{eq:pq bilinear output}). Similar interpretation holds for the terminal-based formulations.}

 In our experiments, we choose the discretization level $H=3$ for all the instances--meaning that the domain of each discretized variable is partitioned into $2^3$ intervals of equal length. Since the solver can take long to close the MILP duality gap, we set time limit of $1800s$ for each instance. If the MILP is not solved within this time limit, the MILP dual bound is taken as dual bound for the corresponding instance. The results for all instances are reported in Tables~\ref{table-MILPgap-mining-inst}, \ref{table-MILPgap-literature-inst} and \ref{table-MILPgap-random-inst}. The best performing method is highlighted in bold.

For the Mining instances, $\mathcal{M}_3^S(H)$ is the best performing method for 19 out of the 24 instances (see Table~\ref{table-MILPgap-mining-inst}). The average run time of $\mathcal{M}_3^S(H)$ is also one of the best.
For the Literature instances, the terminal-based formulation works better. As we can see in Table~\ref{table-MILPgap-literature-inst}, on average, $\mathcal{M}_3^T(H)$ closes almost twice more gap than the second best performing method.
For the Random instances, $\mathcal{M}_3^S(H)$ is again the best performing method. On average, $\mathcal{M}_3^S(H)$ yields the best gap and the best run time (see Table~\ref{table-MILPgap-random-inst}). However, there are instances in which $\mathcal{M}_3^T(H)$ performs significantly better, for example, instances F10 and F11. Thus, it is difficult to advise a single method in this case.

\begin{table}
\centering
\caption{Duality gaps via discretization for Mining instances.}
\label{table-MILPgap-mining-inst}
\resizebox{\textwidth}{!}{
\begin{tabular}{c|cc|cc|cc|cc|cc|cc|}
\multicolumn{1}{c}{}     & \multicolumn{2}{c}{$\mathcal{M}_1^S(H)$} & \multicolumn{2}{c}{$\mathcal{M}_2^S(H)$} & \multicolumn{2}{c}{$\mathcal{M}_3^S(H)$} & \multicolumn{2}{c}{$\mathcal{M}_1^T(H)$} & \multicolumn{2}{c}{$\mathcal{M}_2^T(H)$} & \multicolumn{2}{c}{$\mathcal{M}_3^T(H)$} \\
Inst & Gap               & Time               & Gap               & Time               & Gap               & Time               & Gap               & Time               & Gap               & Time               & Gap               & Time              
\\
\hline
2009H2             & 3.64                 & 1800.17         & 2.69                 & 1800.18         & \textbf{1.61}        & 1800.09         & 2.08                 & 1800.37         & 2.08                 & 1800.27         & 2.65                 & 1800.14         \\
2009Q3             & 0.80                 & 23.67           & 0.54                 & 16.45           & \textbf{0.20}        & 10.64           & 0.21                 & 12.56           & 0.22                 & 18.12           & 0.53                 & 8.56            \\
2009Q4             & 8.82                 & 1800.05         & 8.41                 & 1800.05         & \textbf{3.45}        & 1558.98         & 4.93                 & 1800.06         & 4.74                 & 1800.06         & 4.69                 & 1800.02         \\
2010Y            & 9.13                 & 1800.27         & 8.27                 & 1801.63         & \textbf{8.09}        & 1800.63         & 8.06                 & 1800.33         & 8.24                 & 1800.25         & 8.44                 & 1800.14         \\
2010H1             & 13.31                & 1800.07         & 11.71                & 1800.15         & \textbf{6.92}        & 1800.23         & 9.05                 & 1800.09         & 8.55                 & 1800.23         & 8.41                 & 1800.09         \\
2010H2             & 1.22                 & 1800.11         & \textbf{0.84}        & 1800.11         & 1.03                 & 1800.33         & 1.00                 & 1800.19         & 0.93                 & 1800.11         & 2.19                 & 1800.05         \\
2010Q1             & 2.89                 & 232.87          & 2.75                 & 379.08          & \textbf{1.64}        & 134.11          & 2.25                 & 665.41          & 2.25                 & 1279.26         & 2.32                 & 199.84          \\
2010Q2             & 12.37                & 1028.67         & 11.64                & 1594.03         & \textbf{2.95}        & 232.52          & 4.89                 & 1800.03         & 4.85                 & 1800.05         & 4.23                 & 422.70          \\
2010Q3            & 0.94                 & 39.75           & 0.70                 & 40.09           & \textbf{0.38}        & 14.31           & 0.70                 & 33.23           & 0.70                 & 37.08           & 0.80                 & 48.44           \\
2010Q4            & 6.38                 & 132.55          & 5.28                 & 160.81          & \textbf{1.97}        & 772.50          & 3.84                 & 231.61          & 3.84                 & 239.50          & 4.95                 & 701.95          \\
2011Y            & 1.56                 & 1800.09         & 1.28                 & 1800.19         & 1.20                 & 1800.11         & 1.17                 & 1800.16         & \textbf{1.13}        & 1800.20         & 1.28                 & 1800.84         \\
2011H1            & 0.64                 & 1800.09         & 0.54                 & 1800.08         & \textbf{0.18}        & 871.16          & 0.41                 & 1800.05         & 0.58                 & 1800.14         & 0.31                 & 287.62          \\
2011H2            & 1.37                 & 1800.06         & 1.26                 & 1800.09         & \textbf{0.37}        & 1800.09         & 1.07                 & 1800.09         & 1.05                 & 1800.12         & 0.66                 & 1800.05         \\
2011Q1            & 0.31                 & 18.62           & 0.26                 & 36.94           & \textbf{0.12}        & 57.71           & 0.22                 & 20.05           & 0.22                 & 24.11           & 0.22                 & 29.84           \\
2011Q2            & 1.08                 & 90.28           & 1.01                 & 38.64           & \textbf{0.49}        & 36.70           & 0.89                 & 62.84           & 0.89                 & 111.09          & 0.69                 & 58.31           \\
2011Q3            & 0.05                 & 0.66            & 0.03                 & 0.77            & \textbf{0.01}        & 1.13            & 0.04                 & 1.33            & 0.04                 & 1.81            & 0.06                 & 0.83            \\
2011Q4            & 0.18                 & 26.39           & 0.15                 & 31.48           & \textbf{0.10}        & 8.00            & 0.24                 & 45.20           & 0.24                 & 39.56           & 0.17                 & 49.69           \\
2012Y            & 1.85                 & 1800.06         & 1.43                 & 1800.09         & \textbf{1.20}        & 1802.70         & 1.37                 & 1800.08         & 1.37                 & 1800.06         & 1.95                 & 1800.13         \\
2012H1            & 2.22                 & 1800.03         & 1.72                 & 1800.08         & 1.34                 & 1800.04         & \textbf{1.14}        & 1800.06         & 1.39                 & 1800.08         & 2.48                 & 1800.11         \\
2012H2            & 0.19                 & 39.77           & 0.15                 & 52.55           & \textbf{0.07}        & 52.66           & 0.18                 & 55.59           & 0.18                 & 60.58           & 0.19                 & 72.69           \\
2012Q1            & 3.35                 & 22.00           & 2.15                 & 10.64           & \textbf{0.91}        & 11.08           & 1.32                 & 9.00            & 1.32                 & 23.47           & 1.78                 & 23.52           \\
2012Q2            & 0.16                 & 12.02           & 0.12                 & 10.39           & 0.12                 & 3.97            & 0.16                 & 18.55           & 0.16                 & 19.77           & \textbf{0.11}        & 22.23           \\
2012Q3            & 0.29                 & 19.80           & 0.23                 & 28.28           & \textbf{0.07}        & 28.56           & 0.18                 & 20.72           & 0.18                 & 48.72           & 0.21                 & 33.38           \\
2012Q4            & 1.77                 & 4.73            & 1.65                 & 4.16            & 0.70                 & 1.03            & 0.44                 & 2.31            & 0.44                 & 2.22            & \textbf{0.33}        & 8.56            \\
\hline
Ave. & 3.10	& 820.53	& 2.70	& 850.29	& \textbf{1.46}	& 758.30	& 1.91	& 874.16	& 1.90	& 904.45	& 2.07	& 757.07
\end{tabular}}
\end{table}

\begin{table}[!h]
\centering
\caption{Duality gaps via discretization for Literature instances.}
\label{table-MILPgap-literature-inst}
\resizebox{\textwidth}{!}{
\begin{tabular}{c|cc|cc|cc|cc|cc|cc|}
\multicolumn{1}{c}{}  & \multicolumn{2}{c}{$\mathcal{M}_1^S(H)$} & \multicolumn{2}{c}{$\mathcal{M}_2^S(H)$} & \multicolumn{2}{c}{$\mathcal{M}_3^S(H)$} & \multicolumn{2}{c}{$\mathcal{M}_1^T(H)$} & \multicolumn{2}{c}{$\mathcal{M}_2^T(H)$} & \multicolumn{2}{c}{$\mathcal{M}_3^T(H)$} \\
Inst & Gap               & Time               & Gap               & Time               & Gap               & Time               & Gap               & Time               & Gap               & Time               & Gap               & Time              
\\
\hline
L1            & 0.33                 & 0.03                 & 0.33                 & 0.05                 & \textbf{0.26}        & 0.05                 & 0.36                 & 0.05                 & 0.36                 & 0.05                 & 0.59                 & 0.03                 \\
L2            & 2.96                 & 1.12                 & 2.96                 & 1.70                 & 15.66                & 1.38                 & 2.70                 & 0.20                 & 2.70                 & 0.09                 & \textbf{1.55}        & 0.28                 \\
L3            & 2.96                 & 0.50                 & 2.96                 & 0.86                 & 3.58                 & 0.37                 & 1.97                 & 0.08                 & 1.97                 & 0.09                 & \textbf{1.55}        & 0.17                 \\
L4            & 1.96                 & 16.64                & 1.96                 & 19.78                & 2.45                 & 1.36                 & 0.47                 & 0.49                 & 0.47                 & 0.28                 & \textbf{0.26}        & 2.14                 \\
L5            & 2.88                 & 1.17                 & 2.88                 & 1.58                 & 4.19                 & 12.78                & 0.96                 & 0.81                 & 0.96                 & 0.68                 & \textbf{0.05}        & 0.14                 \\
L6            & \textbf{0.00}        & 0.14                 & \textbf{0.00}        & 0.17                 & \textbf{0.00}        & 0.11                 & \textbf{0.00}        & 0.04                 & \textbf{0.00}        & 0.02                 & \textbf{0.00}        & 0.06                 \\
L12           & \textbf{0.00}        & 0.09                 & \textbf{0.00}        & 0.11                 & \textbf{0.00}        & 0.12                 & \textbf{0.00}        & 0.02                 & \textbf{0.00}        & 0.02                 & \textbf{0.00}        & 0.03                 \\
L13           & \textbf{0.00}        & 0.08                 & \textbf{0.00}        & 0.13                 & \textbf{0.00}        & 0.20                 & \textbf{0.00}        & 0.02                 & \textbf{0.00}        & 0.03                 & \textbf{0.00}        & 0.03                 \\
L14           & \textbf{0.00}        & 0.10                 & \textbf{0.00}        & 0.11                 & 4.76                 & 0.17                 & 1.91                 & 0.02                 & 1.91                 & 0.04                 & \textbf{0.00}        & 0.03                 \\
L15           & 0.77                 & 0.14                 & \textbf{0.68}        & 0.12                 & 1.20                 & 0.08                 & 1.20                 & 0.03                 & 1.20                 & 0.06                 & \textbf{0.68}        & 0.09                 \\
C2            & 0.14                 & 0.64                 & 0.14                 & 0.72                 & \textbf{0.10}        & 0.59                 & \textbf{0.10}        & 0.55                 & \textbf{0.10}        & 0.69                 & 0.14                 & 0.47                 \\
D1            & 1.06                 & 1800.03              & 1.06                 & 1800.00              & 1.04                 & 1800.02              & \textbf{1.02}        & 1800.02              & 1.03                 & 1800.03              & 1.04                 & 1800.02              \\
\hline
Ave. & 1.09 & 151.72 & 1.08 & 152.11 & 2.77 & 151.44 & 0.89 & 150.19 & 0.89 & 150.17 & \textbf{0.49} & 150.29
\end{tabular}}
\end{table}

\begin{table}[!h]
\centering
\caption{Duality gaps via discretization for Random instances.}
\label{table-MILPgap-random-inst}
\resizebox{\textwidth}{!}{
\begin{tabular}{c|cc|cc|cc|cc|cc|cc|}
\multicolumn{1}{c}{}  & \multicolumn{2}{c}{$\mathcal{M}_1^S(H)$} & \multicolumn{2}{c}{$\mathcal{M}_2^S(H)$} & \multicolumn{2}{c}{$\mathcal{M}_3^S(H)$} & \multicolumn{2}{c}{$\mathcal{M}_1^T(H)$} & \multicolumn{2}{c}{$\mathcal{M}_2^T(H)$} & \multicolumn{2}{c}{$\mathcal{M}_3^T(H)$} \\
Inst & Gap               & Time               & Gap               & Time               & Gap               & Time               & Gap               & Time               & Gap               & Time               & Gap               & Time              
\\
\hline
F1            & \textbf{2.04}        & 7.06                 & \textbf{2.04}        & 9.52                & 2.30                 & 4.14                 & 2.30                 & 6.31                 & 2.30                 & 4.13                 & 2.26                 & 10.23                \\
F2            & 3.01                 & 1800.03              & 2.99                 & 1800.02             & \textbf{2.39}        & 1800.08              & 2.63                 & 1800.08              & 2.59                 & 1800.08              & 2.63                 & 1800.03              \\
F3            & 3.42                 & 1800.05              & 4.36                 & 1800.05             & \textbf{2.01}        & 526.69               & 7.50                 & 1800.05              & 8.88                 & 1800.03              & 5.42                 & 1800.09              \\
F4            & \textbf{0.00}        & 2.44                 & \textbf{0.00}        & 3.19                & \textbf{0.00}        & 3.91                 & \textbf{0.00}        & 3.20                 & \textbf{0.00}        & 3.78                 & \textbf{0.00}        & 4.02                 \\
F5            & 7.26                 & 1800.06              & 7.90                 & 1800.03             & 6.52                 & 1800.03              & 6.78                 & 1800.06              & 7.55                 & 1800.05              & \textbf{6.27}        & 1800.06              \\
F6            & 8.25                 & 1800.02              & 8.38                 & 1800.03             & \textbf{8.12}        & 1800.05              & 9.11                 & 1800.05              & 8.29                 & 1800.05              & 8.19                 & 1800.06              \\
F7            & 24.44                & 1800.05              & 25.67                & 1800.06             & 25.96                & 1800.05              & 24.79                & 1800.06              & 26.83                & 1800.04              & \textbf{22.87}       & 1800.08              \\
F8            & 10.40                & 1800.06              & 10.81                & 1800.05             & 9.53                 & 1800.05              & \textbf{7.66}        & 1800.08              & 8.91                 & 1800.06              & 10.55                & 1800.03              \\
F9            & 8.09                 & 1800.06              & 8.29                 & 1800.06             & 6.78                 & 1800.05              & 8.25                 & 1800.06              & 7.88                 & 1800.06              & \textbf{6.37}        & 1800.12              \\
F10           & 3.08                 & 1800.14              & 3.65                 & 1800.14             & 4.23                 & 726.17               & 4.69                 & 764.77               & 4.69                 & 653.10               & \textbf{2.47}        & 1110.62              \\
F11           & 1.25                 & 23.02                & 1.25                 & 26.14               & 1.30                 & 63.05                & 1.47                 & 57.05                & 1.47                 & 71.55                & \textbf{0.84}        & 9.50                 \\
F12           & 1.75                 & 1800.05              & 1.62                 & 1800.07             & \textbf{1.11}        & 431.17               & 1.57                 & 1800.14              & 1.61                 & 1800.12              & 1.49                 & 1800.16              \\
F13           & 4.54                 & 1800.05              & 4.22                 & 1800.05             & \textbf{1.67}        & 1751.19              & 3.81                 & 1800.10              & 3.79                 & 1800.06              & 4.40                 & 1800.05              \\
F14           & \textbf{0.00}        & 1.09                 & \textbf{0.00}        & 0.69                & \textbf{0.00}        & 0.88                 & \textbf{0.00}        & 1.33                 & \textbf{0.00}        & 2.36                 & \textbf{0.00}        & 3.08                 \\
F15           & 13.15                & 1800.03              & 12.73                & 1800.03             & \textbf{12.22}       & 1800.03              & 12.97                & 1800.03              & 13.52                & 1800.05              & 12.43                & 1800.05              \\
F16           & 1.37                 & 240.83               & 1.35                 & 347.73              & 1.16                 & 161.31               & 1.15                 & 573.34               & \textbf{1.09}        & 1460.55              & 1.48                 & 39.92                \\
F17           & 1.94                 & 10.59                & 1.94                 & 19.34               & \textbf{1.82}        & 6.59                 & 1.83                 & 6.95                 & 1.83                 & 10.22                & 2.50                 & 10.50                \\
F18           & 1.79                 & 61.87                & 1.79                 & 42.75               & \textbf{1.16}        & 36.03                & 1.17                 & 21.73                & 1.17                 & 27.42                & 2.64                 & 16.83                \\
F19           & \textbf{0.00}        & 0.75                 & \textbf{0.00}        & 0.56                & \textbf{0.00}        & 0.38                 & \textbf{0.00}        & 0.39                 & \textbf{0.00}        & 0.39                 & \textbf{0.00}        & 0.44                 \\
F20           & 3.47                 & 1800.11              & 3.50                 & 1800.05             & 2.49                 & 535.50               & 2.29                 & 631.09               & 2.29                 & 1337.64              & \textbf{2.23}        & 308.12               \\
F21           & 26.01                & 1800.05              & \textbf{25.46}       & 1800.08             & 25.55                & 1800.06              & 26.20                & 1800.05              & 25.69                & 1800.09              & 25.46                & 1800.11              \\
F22           & 10.40                & 1800.03              & 11.68                & 1800.03             & 11.11                & 1800.05              & 12.60                & 1800.04              & 12.95                & 1800.08              & 7.26                 & 1800.13              \\
F23           & 43.65                & 1800.06              & \textbf{43.58}       & 1800.11             & \textbf{43.58}       & 1800.09              & 46.52                & 1800.05              & 45.03                & 1800.09              & 44.69                & 1800.09              \\
F24           & 5.96                 & 46.45                & 5.96                 & 39.28               & \textbf{5.88}        & 12.58                & 5.99                 & 12.37                & 5.99                 & 28.78                & 5.92                 & 20.03                \\
\hline
Ave. & 7.72 & 1141.46 & 7.88 & 1145.42 & \textbf{7.37} & 927.50 & 7.97 & 1061.64 & 8.10 & 1125.03 & 7.43 & 1038.93
\end{tabular}}
\end{table}

In Table~\ref{table:LPvsMILP}, we compare the performances of the average best LP and MILP methods on each instance set. 
As expected, the average run time of each MILP method was much higher than its LP counterpart. On the other hand, MILP methods can close significantly more gap than the LP ones. The performance discrepancy is more evident in the Literature instances due to their relatively small sizes. Another interesting observation is that the winning Light LP method seems to suggest which MILP method will perform better. For instance, the source-based MILP relaxations perform better for the Mining and Random instances, as correctly predicted by the better performance of the source-based light LP relaxations. This situation is reversed for the Literature instances as the terminal-based light LP and MILP relaxations seem to provide stronger relaxations consistently.

\begin{table}[h]
\centering
\caption{Best average duality gap for each set of instances. Here, ``Method'' is the method that yields the best average duality gap.}
\label{table:LPvsMILP}
\begin{tabular}{l|ccc|ccc|ccc|}
\multicolumn{1}{c}{}         & \multicolumn{3}{c}{Light LP}  & \multicolumn{3}{c}{LP} & \multicolumn{3}{c}{MILP} \\
Instance Set & Gap      & Time  & Method  & Gap      & Time  & Method     & Gap      & Time   & Method     \\
\hline
Mining     & 4.14  & 1.76 & $\mathcal{F}_2^S$  & 3.82       & 13.73   & $\mathcal{F}_2^S\cap\mathcal{F}_2^T$   & 1.46       & 758.30    & $\mathcal{M}_3^S(H)$   \\
Literature   & 18.11  & 0.02 & $\mathcal{F}_2^T$ & 18.10       & 0.11   & $\mathcal{F}_2^S\cap\mathcal{F}_2^T$   & 0.49       & 150.29   & $\mathcal{M}_3^T(H)$    \\
Random     & 12.47  & 5.25 & $\mathcal{F}_2^S$  & 12.28      & 27.34  & $\mathcal{F}_2^S\cap\mathcal{F}_2^T$   & 7.37        & 927.50    & $\mathcal{M}_3^S(H)$   
\end{tabular}
\end{table}

\section{Conclusion}

We propose new convex relaxations for QCQPs derived from its rank-based formulation \eqref{eq:QCQP2a}--\eqref{eq:QCQP2d}. Specifically, we study the convex hull of sets defined by a rank-1 constraint intersected with some linear side constraints \eqref{eq:cU}. For several choices of linear side constraints, we show that this convex hull is polyhedral or SOCr, and provide compact formulations for the polyhedral cases.  We also show that in all these cases, a linear objective can be optimized over these sets in polynomial time. 

On the application side, we propose rank-1 based formulations for the pooling problem. The new formulations combined with our convexification results allow us to derive new convex relaxations for the pooling problem, which we show to generalize, for example, the well-known $pq$-relaxation. {Studying the pooling problem via rank-based formulations not only allows us to recover previous relaxations from the literature, but also leads us to improve/strengthen them in a systematic way.} 
 In addition, inspired by our newly proposed formulation and our convexification results, we propose several MILP restriction and  relaxation discretizations to the pooling problem. 

Finally, we report extensive computational experiments on three sets of generalized pooling problem instances, two from the literature and one introduced in this paper. The new set of pooling problem instances being introduced here was randomly generated and are relatively harder to solve than all the previously available instances in the literature, therefore, it may serve as a new benchmark for new methodologies. Our computational results show that our technique consistently outperforms, on average, the previous methods from the literature in deriving dual bounds for pooling problem instances.



\appendix

\section{Omitted proofs}
\label{app:proofs}

\subsection{Proof of Part (iii) of Proposition~\ref{prop:row} }\label{sec:proofprop3}


\begin{proposition}
$\textup{conv} \left({\cU}^{\textup{row}}_{(n_1,n_2)}(l,u)\right)  $ is described by the inequalities (\ref{eq:rowconv1}), (\ref{eq:rowconv2}), (\ref{eq:rowconv3}), and (\ref{eq:rowconv4}).
\end{proposition}
\begin{proof}
We will use Fourier-Motzkin elimination to obtain the convex hull in the original space.  Now, we will project the $t$ variables in the order $t_{n_2}$, $t_{n_2-1}$, $t_{n_2 -2}$, \dots, $t_1$.  

We claim that after projecting out variables $t_{n_2}, \dots, t_{n_2 - j}$, the resulting system of the inequalities is:
\begin{align*}
1 - \sum_{p = 1}^{n_2 -  (j + 1)} t_p \leq& \ \sum_{i \in \mathcal{I}} \sum_{k \in T_{i} }  \frac{W_{ik}}{l_i} \ &\forall &  (T_1, T_2, \dots, T_ {|\mathcal{I}|}) \in \cP_{|\cI|}(\{n_2 - j, \dots, n_2\}) \\
1 - \sum_{p = 1}^{n_2 -  (j + 1)} t_p \geq& \ \sum_{i \in [n_1]} \sum_{k \in T_i}  \frac{W_{ik}}{u_i}\ &\forall&  (T_1, T_2, \dots, T_ {n_1}) \in \cP_{n_1}(\{n_2 - j, \dots, n_2\}) \\
u_{i_2}W_{i_1k} \geq \ & l_{i_1} W_{i_2k} &\forall& i_2 \in [n_1], \forall i_1 \in \mathcal{I}, \forall k \in  \{n_2 - j, \dots, n_2\} \\
 l_it_k  \le  W_{ik} \leq \ & u_it_k &\forall& i \in [n_1], \forall k \in [n_2 - (j + 1)] \nonumber\\
t_k \geq \ & 0 \ &\forall&  k \in [n_2 - (j + 1)] \nonumber \\
W_{ij} \geq \ & 0 \ &\forall&  i \in [n_1], \forall j \in [n_2 ].
\end{align*}

\textbf{Base case:} After projecting out $t_{n_2}$, we obtain the system:
\begin{align}
\sum_{j = 1}^{n_2 - 1} t_j \leq& \ 1 \label{eq:rowbase1}\\
W_{i n_2} \leq& \ (1 - \sum_{j = 1}^{n_2 - 1} t_j)u_i \ &\forall& i \in [n_1] \label{eq:rowbase2}\\
W_{i n_2} \geq& \ (1 - \sum_{j = 1}^{n_2 - 1} t_j)l_i \ &\forall& i \in \mathcal{I} \nonumber\\
u_{i_2}W_{i_1n_2} \geq& \ l_{i_1} W_{i_2n_2} &\forall& i_2 \in [n_1],\forall  i_1 \in \mathcal{I} \nonumber \\
 l_it_j  \le W_{ij} \leq& \ u_it_j &\forall& i \in [n_1], \forall j \in [n_2 - 1] \nonumber\\
t_j  \geq& \ 0 \ &\forall&  j \in [n_2 - 1], \nonumber \\
W_{ij} \geq& \  0 \ &\forall&  i \in [n_1], \forall  j \in [n_2 ],  \label{eq:rowbase6}
\end{align}
Note that (\ref{eq:rowbase2}) and (\ref{eq:rowbase6}) imply (\ref{eq:rowbase1}). Therefore the above can be written as:
\begin{align}
(1 - \sum_{j = 1}^{n_2 - 1} t_j) \leq & \  \sum_{i \in \mathcal{I}} \sum_{k \in T_{i} }  \frac{W_{in_2}}{l_i}    \ &\forall &  (T_1, T_2, \dots, T_ {|\mathcal{I}|}) \in \cP_{|\cI|}(\{n_2 \}) \nonumber\\
(1 - \sum_{j = 1}^{n_2 - 1} t_j) \geq & \ \sum_{i \in [n_1]} \sum_{k \in T_i}  \frac{W_{ik}}{u_i}    \ &\forall&  (T_1, T_2, \dots, T_ {n_1}) \in \cP_{n_1}(\{n_2 \})  \nonumber\\
u_{i_2}W_{i_1n_2} \geq& \ l_{i_1} W_{i_2n_2} &\forall& i_2 \in [n_1], \forall  i_1 \in \mathcal{I} \nonumber \\
 l_it_j  \le W_{ij} \leq& \ u_it_j &\forall& i \in [n_1], \forall  j \in [n_2 - 1] \nonumber\\
t_j  \geq& \ 0 \ &\forall&  j \in [n_2 - 1], \nonumber \\
W_{ij} \geq& \ 0 \ &\forall&  i \in [n_1],  \forall  j \in [n_2],  \nonumber
\end{align}
proving the base case.
\newline\textbf{Induction step:} After projecting $t_{n_2}$, $\dots$, $t_{n_2 -j}$, by the induction hypothesis we  have the following system:
\begin{align}
1 - \sum_{p = 1}^{n_2 -  (j + 2)} t_p - \sum_{i \in \mathcal{I}} \sum_{k \in T_{i} } \frac{W_{ik}}{l_i} \leq& \  t_{n_2 - (j +1)}  \ &\forall&  (T_1, T_2, \dots, T_ {|\mathcal{I}|}) \in \cP_{|\cI|}(\{n_2 - j, \dots, n_2\})  \nonumber\\
\frac{W_{i,n_2 - (j +1)}}{u_i} \leq & \ t_{n_2 - (j +1)} &\forall& i \in [n_1] \nonumber \\
1 - \sum_{p = 1}^{n_2 -  (j + 2)} t_p  - \sum_{i \in [n_1]} \sum_{k \in T_i}  \frac{W_{ik}}{u_i} \geq& \  t_{n_2 - (j +1)} \ &\forall&  (T_1, T_2, \dots, T_ {n_1}) \in \cP_{n_1}(\{n_2 - j, \dots, n_2\}) \nonumber \\
\frac{W_{i,n_2 - (j +1)}}{l_i} \geq& \ t_{n_2 - (j +1)} &\forall& i \in \mathcal{I} \nonumber \\
u_{i_2}W_{i_1k} \geq& \ l_{i_1} W_{i_2k} &\forall& i_2 \in [n_1], \forall  i_1 \in \mathcal{I},\forall  k \in  \{n_2 - j, \dots, n_2\} \nonumber \\
l_it_k  \le W_{ik} \leq& \ u_it_k &\forall& i \in [n_1], \forall k \in [n_2 - (j + 1)] \nonumber\\
t_k \geq& \ 0 \ &\forall&  k \in [n_2 - (j + 1)] \nonumber \\
W_{ik} \geq& \ 0 \ &\forall&  i \in [n_1], \forall k \in [n_2 ]. \nonumber 
\end{align}
Note that a constraint of the form:
$$- \sum_{i \in [n_1]} \sum_{k \in T_i}  \frac{W_{ik}}{u_i}\geq- \sum_{i \in \mathcal{I}} \sum_{k \in T'_{i} }\frac{W_{ik}}{l_i},$$
where $(T_1, T_2 \dots T_ {n_1}) \in \cP_{n_1}(\{n_2 - j, \dots, n_2\})$ and $(T'_1,  T'_2, \dots T'_ {|\mathcal{I}|})\in \cP_ {|\mathcal{I}|} (\{n_2 - j, \dots, n_2\})$ is implied by constraints of the form $u_{i_2}W_{i_1k} \geq l_{i_1} W_{i_2k} \forall i_2 \in [n_1], i_1 \in \mathcal{I}, k \in  \{n_2 - j, \dots, n_2\}.$ Thus, after projecting $t_{n_2 - (j +1)}$, we obtain:
\begin{align}
1 - \sum_{p = 1}^{n_2 -  (j + 2)} t_p \leq& \ \sum_{i \in \mathcal{I}} \sum_{k \in T_{i} }  \frac{W_{ik}}{l_i} \  &\forall&  (T_1, T_2, \dots, T_ {|\mathcal{I}|}) \in \cP_ {|\mathcal{I}|}( \{n_2 - (j +1), \dots, n_2\} )\nonumber \\
1 - \sum_{p = 1}^{n_2 -  (j + 2)} t_p  \geq& \ \sum_{i \in [n_1]} \sum_{k \in T_i}  \frac{W_{ik}}{u_i}\  &\forall& (T_1, T_2, \dots, T_ {n_1}) \in \cP_{n_1} (\{n_2 - (j +1), \dots, n_2\}) \nonumber \\ 
u_{i_2}W_{i_1k} \geq& \ l_{i_1} W_{i_2k}  &\forall& i_2 \in [n_1], \forall i_1 \in \mathcal{I}, \forall k \in  \{n_2 - (j + 1), \dots, n_2\} \nonumber \\
 l_it_k  \le W_{ik} \leq& \ u_it_k  &\forall& i \in [n_1], \forall k \in [n_2 - (j + 2)] \nonumber\\
t_k \geq& \ 0 \  &\forall&  k \in [n_2 - (j + 2)] \nonumber \\
W_{ik} \geq& \ 0 \  &\forall&  i \in [n_1], \forall k \in [n_2 ]. \nonumber
\end{align}
It is straightforward now to see that after all $t$ variables are projected, we obtain the result.
\end{proof}

\begin{proposition}\label{prop:sep1}
The inequalities in (\ref{eq:rowconv1}) can be separated in polynomial-time. 
\end{proposition}
\begin{proof}
For a given matrix $\hat W$, let us define an index $j_{\text{row}}$ for each index $j \in [n_2]$ as
\begin{equation}\label{eq:jRow def}
j_{\text{row}}:= \min \left( \argmax_{i =1,\dots,n_1} \left \{ \frac{\hat W_{ij}}{u_i}    \right\} \right).
\end{equation}
Here, we are breaking ties arbitrarily using the smallest index, when necessary. Then, we define a partition $T_1^*,\dots,T_{n_1}^*$ of the set $ [n_2]$ as 
\[
T_i^* := \{j \,|\, j_{\text{row}} = i\} .
\]
Let 
\[
\theta := \sum_{i=1}^{n_1}\sum_{j \in T_i^*} \frac{\hat W_{ij}}{u_i} .
\]
If $\theta >1$, then a violated inequality is discovered. Otherwise, we conclude that $\hat W$ satisfies all the inequalities in (\ref{eq:rowconv1}) (by construction, the partition $T_1^*, \dots, T_{n_1}^*$ corresponds to the inequality with the largest deviation, if one exists). 
Finally, we note that the complexity of this separation routine is  $\mathcal{O}(n_1n_2)$ since we need to find the maximum of $n_1$ numbers $n_2$ times to construct this partition.
\end{proof}

\begin{proposition}
The inequalities in (\ref{eq:rowconv2}) can be separated in polynomial-time. 
\end{proposition}
\begin{proof}
The proof is similar to the proof of Proposition \ref{prop:sep1}.
\end{proof}

\subsection{Proof of Part (ii)  Proposition~\ref{prop:rowplus}}\label{sec:proofprop4}

\begin{proposition}
$\textup{conv} \left({\cU}^{\textup{row+}}_{(n_1,n_2)}(l,u, L, U)\right)  $ is described by the inequalities (\ref{eq:rowplusconv1}), (\ref{eq:rowplusconv2}), (\ref{eq:rowconv3}), and (\ref{eq:rowconv4}).
\end{proposition}

\begin{proof}
We will use Fourier-Motzkin elimination to obtain the convex hull in the original space.  Now, we will project the $t$ variables in the order $t_{n_2}$, $t_{n_2-1}$, $t_{n_2 -2}$, \dots, $t_1$.  

We claim that after projecting out variables $t_{n_2}, \dots, t_{n_2 - j}$, the resulting system of the inequalities is:
\begin{align*}
1 - \sum_{p = 1}^{n_2 -  (j + 1)} t_p \leq& \ \sum_{i \in \mathcal{I}} \sum_{k \in T_{i} }  \frac{W_{ik}}{l_i} +   \frac1L \sum_{i = 1}^{n_1}\sum_{k \in T_0}  W_{ik}  \ &\forall &  (T_0, T_1, \dots, T_ {|\mathcal{I}|}) \in \cP_{|\cI|}(\{n_2 - j, \dots, n_2\}) \\
1 - \sum_{p = 1}^{n_2 -  (j + 1)} t_p \geq& \ \sum_{i \in [n_1]} \sum_{k \in T_i}  \frac{W_{ik}}{u_i} +   \frac1U \sum_{i = 1}^{n_1}\sum_{k \in T_0}  W_{ik}  \ &\forall&  (T_0, T_1, \dots, T_ {n_1}) \in \cP_{n_1}(\{n_2 - j, \dots, n_2\}) \\
u_{i_2}W_{i_1k} \geq \ & l_{i_1} W_{i_2k} &\forall& i_2 \in [n_1], \forall i_1 \in \mathcal{I}, \forall k \in  \{n_2 - j, \dots, n_2\} \\
 l_it_k  \le W_{ik} \leq \ & u_it_k &\forall& i \in [n_1], \forall k \in [n_2 - (j + 1)] \nonumber\\
L t_k \le \sum_{i=1}^{n_1}W_{i k} \leq& \ U t_k\ &\forall&     k \in [n_2- (j + 1)] \nonumber\\
t_k \geq \ & 0 \ &\forall&  k \in [n_2 - (j + 1)] \nonumber \\
W_{ik} \geq \ & 0 \ &\forall&  i \in [n_1], \forall k\in [n_2 ].
\end{align*}

%
%

\textbf{Base case:} After projecting out $t_{n_2}$, we obtain the system:
\begin{align}
\sum_{j = 1}^{n_2 - 1} t_j \leq& \ 1 \label{eq:rowplusbase1}\\
W_{i n_2} \leq& \ (1 - \sum_{j = 1}^{n_2 - 1} t_j)u_i \ &\forall& i \in [n_1] \label{eq:rowplusbase2}\\
W_{i n_2} \geq& \ (1 - \sum_{j = 1}^{n_2 - 1} t_j)l_i \ &\forall& i \in \mathcal{I} \nonumber\\
\sum_{i=1}^{n_1}W_{i n_2} \leq& \ (1 - \sum_{j = 1}^{n_2 - 1} t_j)U \   \nonumber\\
\sum_{i=1}^{n_1}W_{i n_2} \geq& \ (1 - \sum_{j = 1}^{n_2 - 1} t_j)L \   \nonumber\\
u_{i_2}W_{i_1n_2} \geq& \ l_{i_1} W_{i_2n_2} &\forall& i_2 \in [n_1],\forall  i_1 \in \mathcal{I} \nonumber \\
l_it_j  \le W_{ij} \leq& \ u_it_j &\forall& i \in [n_1], \forall j \in [n_2 - 1] \nonumber\\
 L t_j  \le \sum_{i=1}^{n_1}W_{i j} \leq& \ U t_j \ &\forall&     j \in [n_2-1] \nonumber\\
t_j  \geq& \ 0 \ &\forall&  j \in [n_2 - 1], \nonumber \\
W_{ij} \geq& \  0 \ &\forall&  i \in [n_1], \forall  j \in [n_2 ],  \label{eq:rowplusbase6}
\end{align}
Note that (\ref{eq:rowplusbase2}) and (\ref{eq:rowplusbase6}) imply (\ref{eq:rowplusbase1}). Therefore the above can be written as:
\begin{align}
(1 - \sum_{j = 1}^{n_2 - 1} t_j) \leq & \  \sum_{i \in \mathcal{I}} \sum_{k \in T_{i} }  \frac{W_{in_2}}{l_i} +   \frac1L \sum_{i = 1}^{n_1}\sum_{k \in T_0}  W_{ik}  \ &\forall &  (T_0, T_1, \dots, T_ {|\mathcal{I}|}) \in \cP_{|\cI|}(\{n_2 \}) \nonumber\\
(1 - \sum_{j = 1}^{n_2 - 1} t_j) \geq & \ \sum_{i \in [n_1]} \sum_{k \in T_i}  \frac{W_{ik}}{u_i} +   \frac1U \sum_{i = 1}^{n_1}\sum_{k \in T_0}  W_{ik}  \ &\forall&  (T_0, T_1, \dots, T_ {n_1}) \in \cP_{n_1}(\{n_2 \})  \nonumber\\
u_{i_2}W_{i_1n_2} \geq& \ l_{i_1} W_{i_2n_2} &\forall& i_2 \in [n_1], \forall  i_1 \in \mathcal{I} \nonumber \\
l_it_j  \le W_{ij} \leq& \ u_it_j &\forall& i \in [n_1], \forall  j \in [n_2 - 1] \nonumber\\
L t_j  \le \sum_{i=1}^{n_1}W_{i j} \leq& \ U t_j \ &\forall&     j \in [n_2-1] \nonumber\\
t_j  \geq& \ 0 \ &\forall&  j \in [n_2 - 1], \nonumber \\
W_{ij} \geq& \ 0 \ &\forall&  i \in [n_1],  \forall  j \in [n_2],  \nonumber
\end{align}
proving the base case.
\newline\textbf{Induction step:} After projecting $t_{n_2}$, $\dots$, $t_{n_2 -j}$, by the induction hypothesis we  have the following system:
\begin{align}
1 - \sum_{p = 1}^{n_2 -  (j + 2)} t_p - \sum_{i \in \mathcal{I}} \sum_{k \in T_{i} } \frac{W_{ik}}{l_i} - \frac1L \sum_{i = 1}^{n_1}\sum_{k \in T_0}  W_{ik} \leq& \  t_{n_2 - (j +1)}  \ &\forall&  (T_0, \dots, T_ {|\mathcal{I}|}) \in \cP_{|\cI|}(\{n_2 - j, \dots, n_2\})  \nonumber\\
\frac{W_{i,n_2 - (j +1)}}{u_i} \leq & \ t_{n_2 - (j +1)} &\forall& i \in [n_1] \nonumber \\
\frac1U\sum_{i=1}^{n_1} {W_{i,n_2 - (j +1)}} \leq & \ t_{n_2 - (j +1)} &\forall& i \in [n_1] \nonumber \\
1 - \sum_{p = 1}^{n_2 -  (j + 2)} t_p  - \sum_{i \in [n_1]} \sum_{k \in T_i}  \frac{W_{ik}}{u_i} -   \frac1U \sum_{i = 1}^{n_1}\sum_{k \in T_0}  W_{ik} \geq& \  t_{n_2 - (j +1)} \ &\forall&  (T_0, \dots, T_ {n_1}) \in \cP_{n_1}(\{n_2 - j, \dots, n_2\}) \nonumber \\
\frac{W_{i,n_2 - (j +1)}}{l_i} \geq& \ t_{n_2 - (j +1)} &\forall& i \in \mathcal{I} \nonumber \\
\frac1L\sum_{i=1}^{n_1} {W_{i,n_2 - (j +1)}} \geq & \ t_{n_2 - (j +1)} &\forall& i \in [n_1] \nonumber \\
u_{i_2}W_{i_1k} \geq& \ l_{i_1} W_{i_2k} &\forall& i_2 \in [n_1], \forall  i_1 \in \mathcal{I},\forall  k \in  \{n_2 - j, \dots, n_2\} \nonumber \\
 l_it_k  \le W_{ik} \leq& \ u_it_k &\forall& i \in [n_1], \forall k \in [n_2 - (j + 1)] \nonumber\\
L t_j  \le \sum_{i=1}^{n_1}W_{i j} \leq& \ U t_j \ &\forall&     j \in [n_2-(j + 1)] \nonumber\\
t_k \geq& \ 0 \ &\forall&  k \in [n_2 - (j + 1)] \nonumber \\
W_{ik} \geq& \ 0 \ &\forall&  i \in [n_1], \forall k \in [n_2 ]. \nonumber 
\end{align}
Note that a constraint of the form:
$$- \sum_{i \in [n_1]} \sum_{k \in T_i}  \frac{W_{ik}}{u_i} - \frac1U \sum_{i \in [n_1]} \sum_{k \in T_0}  {W_{ik}}  \geq- \sum_{i \in \mathcal{I}} \sum_{k \in T'_{i} }\frac{W_{ik}}{l_i} - \frac1L \sum_{i \in \cI} \sum_{k \in T_0}  {W_{ik}} ,$$
where $(T_0, T_1 \dots T_ {n_1}) \in \cP_{n_1}(\{n_2 - j, \dots, n_2\})$ and $(T'_0,  T'_1, \dots T'_ {|\mathcal{I}|})\in \cP_ {|\mathcal{I}|} (\{n_2 - j, \dots, n_2\})$ is implied by constraints of the form $u_{i_2}W_{i_1k} \geq l_{i_1} W_{i_2k} \forall i_2 \in [n_1], i_1 \in \mathcal{I}, k \in  \{n_2 - j, \dots, n_2\}$, and the fact that $L \le U$. Thus, after projecting $t_{n_2 - (j +1)}$, we obtain:
\begin{align}
1 - \sum_{p = 1}^{n_2 -  (j + 2)} t_p \leq& \ \sum_{i \in \mathcal{I}} \sum_{k \in T_{i} }  \frac{W_{ik}}{l_i} + \frac1L \sum_{i \in \mathcal{I}} \sum_{k \in T_{0} }  {W_{ik}} \  &\forall&  (T_1, T_2, \dots, T_ {|\mathcal{I}|}) \in \cP_ {|\mathcal{I}|}( \{n_2 - (j +1), \dots, n_2\} )\nonumber \\
1 - \sum_{p = 1}^{n_2 -  (j + 2)} t_p  \geq& \ \sum_{i \in [n_1]} \sum_{k \in T_i}  \frac{W_{ik}}{u_i} + \frac1U \sum_{i \in \mathcal{I}} \sum_{k \in T_{0} }  {W_{ik}} \  &\forall& (T_1, T_2, \dots, T_ {n_1}) \in \cP_{n_1} (\{n_2 - (j +1), \dots, n_2\}) \nonumber \\ 
u_{i_2}W_{i_1k} \geq& \ l_{i_1} W_{i_2k}  &\forall& i_2 \in [n_1], \forall i_1 \in \mathcal{I}, \forall k \in  \{n_2 - (j + 1), \dots, n_2\} \nonumber \\
l_it_k  \le W_{ik} \leq& \ u_it_k  &\forall& i \in [n_1], \forall k \in [n_2 - (j + 2)] \nonumber\\
L t_k  \le \sum_{i=1}^{n_1} W_{ik} \leq& \ U t_k  &\forall&  \forall k \in [n_2 - (j + 2)] \nonumber\\
t_k \geq& \ 0 \  &\forall&  k \in [n_2 - (j + 2)] \nonumber \\
W_{ik} \geq& \ 0 \  &\forall&  i \in [n_1], \forall k \in [n_2 ]. \nonumber
\end{align}
It is straightforward now to see that after all $t$ variables are projected, we obtain the result.
\end{proof}

{\color{red}  }

\begin{proposition}\label{prop:sepplus1}
The inequalities in (\ref{eq:rowplusconv1}) can be separated in polynomial-time. 
\end{proposition}
\begin{proof}
For a given matrix $\hat W$, let us define an index $j_{\text{row+}}$ for each index $j \in \{1,\dots,n_2\}$ as
\[
j_{\text{row+}} := \begin{cases} 0 & \text{ if }  \frac1U \sum_{i=1}^{n_1} \hat W_{ij} \ge \max_{i =1,\dots,n_1} \left \{ \frac{\hat W_{ij}}{u_i}  \right\}  \\
j_{\text{row}}  & \text{ otherwise}
\end{cases},
\]
where $j_{\text{row}}$ is defined according to \eqref{eq:jRow def}.

Then, we define a partition $T_0^*, T_1^*,\dots,T_{n_1}^*$ of the set $ [n_2]$ as 
\[
T_i^* := \{j \,|\, j_{\text{row+}}= i\} .
\]
Let 
\[
\theta := \sum_{i=1}^{n_1}\sum_{j \in T_i^*} \frac{\hat W_{ij}}{u_i} +   \frac1U \sum_{i = 1}^{n_1}\sum_{j \in T_0^*} \hat W_{ij}  .
\]
If $\theta >1$, then a violated inequality is discovered. Otherwise, we conclude that $\hat W$ satisfies all the inequalities in (\ref{eq:rowplusconv1}) (by construction, the partition $T_0^*, T_1^*, \dots, T_m^*$ corresponds to the inequality with the largest deviation, if one exists). 
Finally, we note that the complexity of this separation routine is again  $\mathcal{O}(n_1n_2)$.
\end{proof}


\begin{proposition}
The inequalities in (\ref{eq:rowplusconv2}) can be separated in polynomial-time. 
\end{proposition}
\begin{proof}
The proof is similar to the proof of Proposition \ref{prop:sepplus1}.
\end{proof}

\subsection{Proof of Theorem~\ref{thm:2by2}}\label{sec:proofthm4}
The proof of  SOC-representability follows due to \cite{santana2018convex} as the convex hull of a set described by the quadratic constraint $W_{11} W_{22} = W_{21} W_{12}$  and  bound constraints is SOCr.

We next present an example where the convex hull of ${\cU}^{\textup{row}}_{(2,2)}(l,u) \cap {\cU}^{\textup{col}}_{(2,2)}(l,u)$ is not polyhedral.
\begin{proposition}
A point of the form:$$\left[\begin{array}{cc} a &\frac{a^2}{1 -a}\\ 1 - a & a \end{array}\right],$$
for $a \in [0,1)$ is an extreme point of the set ${\cU}^{\textup{row}}_{(2,2)}(l,u) \cap {\cU}^{\textup{col}}_{(2,2)}(l,u)$ where $l = (0,1)$ and $u= (1,1)$.
\end{proposition}
\begin{proof}
Clearly the point is feasible. Also if it is not extreme, then it should be possible to write $(a, \frac{a^2}{1 - a}) \in \mathbb{R}^2$ as a convex combination of points of the form $(a_i, \frac{a_i^2}{1 - a_i}) \in \mathbb{R}^2)$ with $a_i \in [0,1)\setminus \{a\}$. However, since $f(u) = \frac{u^2}{1 -u}$ is a strictly convex function, this is not possible (this is because, for example, $f(u) > \frac{a^2}{1 - a} + \left(\frac{1}{(1-a)^2} - 1\right)(u - a) $ for all $u \neq a$, while $f(u) = \frac{a^2}{1 - a} + \left(\frac{1}{(1-a)^2} - 1\right)(u - a) $ for $u = a$). 
\end{proof}

\section{Instance description}
\label{app:instances}

In Tables~\ref{table_size_coal_inst}, \ref{table_size_literature_inst} and \ref{table_size_random_inst}, $AIL$ denotes the subset of arcs $A \cap (I\times L)$. The sets $ALL, ALJ$ and $AIJ$ are defined analogously. The column \textit{Avg. Size $x^s$} (resp. \textit{Avg. Size $x^t$}) displays the average size, over all pools $i\in L$, of the variable matrices $[x^s_{ij}]_{(s,j)}$ (resp. $[x^t_{ij}]_{(i,t)}$) for each instance.

\begin{table}[!h]
\centering
\caption{Mining instances description.}
\label{table_size_coal_inst}
\resizebox{\textwidth}{!}{
\begin{tabular}{cccccccccccccccc}
Inst & $|I|$ & $|L|$ & $|J|$ & $|A|$ & $|AIL|$ & $|ALL|$ & $|ALJ|$ & $|AIJ|$ & $|K|$ & $|x^s|$ & Avg. Size $x^s$ & $|q^s|$ & $|x^t|$ & Avg. Size $x^t$ & $|q^t|$ \\
\hline
2009H2    & 73      & 73      & 50      & 244     & 73        & 71        & 100       & 0         & 4       & 3205    & (18.75, 2.34)   & 1369    & 3718    & (1.97, 26.15)   & 1909    \\
2009Q3    & 31      & 31      & 22      & 104     & 31        & 29        & 44        & 0         & 4       & 594     & (8.26, 2.35)    & 256     & 694     & (1.94, 11.90)   & 369     \\
2009Q4    & 38      & 38      & 27      & 128     & 38        & 36        & 54        & 0         & 4       & 905     & (10.00, 2.37)   & 380     & 1072    & (1.95, 14.82)   & 563     \\
2010Y    & 170     & 170     & 123     & 584     & 170       & 168       & 246       & 0         & 4       & 18038   & (43.05, 2.44)   & 7319    & 21532   & (1.99, 64.05)   & 10889   \\
2010H1   & 86      & 86      & 64      & 298     & 86        & 84        & 128       & 0         & 4       & 4785    & (22.05, 2.47)   & 1896    & 5822    & (1.98, 34.59)   & 2975    \\
2010H2    & 84      & 84      & 59      & 284     & 84        & 82        & 118       & 0         & 4       & 4540    & (21.51, 2.38)   & 1807    & 5516    & (1.98, 33.54)   & 2817    \\
2010Q1    & 39      & 39      & 29      & 134     & 39        & 37        & 58        & 0         & 4       & 1068    & (10.26, 2.44)   & 400     & 1356    & (1.95, 18.13)   & 707     \\
2010Q2    & 43      & 43      & 31      & 146     & 43        & 41        & 62        & 0         & 4       & 1196    & (11.30, 2.40)   & 486     & 1444    & (1.95, 17.51)   & 753     \\
2010Q3   & 39      & 39      & 25      & 126     & 39        & 37        & 50        & 0         & 4       & 914     & (10.26, 2.23)   & 400     & 1056    & (1.95, 14.18)   & 553     \\
2010Q4   & 43      & 43      & 32      & 148     & 43        & 41        & 64        & 0         & 4       & 1264    & (11.26, 2.44)   & 484     & 1582    & (1.95, 19.14)   & 823     \\
2011Y   & 121     & 121     & 95      & 430     & 121       & 119       & 190       & 0         & 4       & 9706    & (30.75, 2.55)   & 3721    & 12022   & (1.98, 50.46)   & 6106    \\
2011H1   & 67      & 67      & 50      & 232     & 67        & 65        & 100       & 0         & 4       & 2811    & (17.25, 2.46)   & 1156    & 3344    & (1.97, 25.70)   & 1722    \\
2011H2   & 53      & 53      & 43      & 190     & 53        & 51        & 86        & 0         & 4       & 1993    & (13.75, 2.58)   & 729     & 2548    & (1.96, 24.85)   & 1317    \\
2011Q1   & 35      & 35      & 27      & 122     & 35        & 33        & 54        & 0         & 4       & 784     & (9.26, 2.49)    & 324     & 936     & (1.94, 14.14)   & 495     \\
2011Q2   & 30      & 30      & 22      & 102     & 30        & 28        & 44        & 0         & 4       & 585     & (8.03, 2.40)    & 241     & 704     & (1.93, 12.47)   & 374     \\
2011Q3   & 19      & 19      & 15      & 66      & 19        & 17        & 30        & 0         & 4       & 260     & (5.26, 2.47)    & 100     & 328     & (1.89, 9.42)    & 179     \\
2011Q4   & 28      & 28      & 23      & 100     & 28        & 26        & 46        & 0         & 4       & 566     & (7.50, 2.57)    & 210     & 722     & (1.93, 13.71)   & 384     \\
2012Y   & 107     & 107     & 79      & 370     & 107       & 105       & 158       & 0         & 4       & 7394    & (27.31, 2.46)   & 2922    & 9000    & (1.98, 42.79)   & 4579    \\
2012H1   & 65      & 65      & 46      & 220     & 65        & 63        & 92        & 0         & 4       & 2719    & (16.85, 2.38)   & 1095    & 3286    & (1.97, 25.98)   & 1689    \\
2012H2   & 41      & 41      & 32      & 144     & 41        & 39        & 64        & 0         & 4       & 1092    & (10.76, 2.51)   & 441     & 1320    & (1.95, 16.88)   & 692     \\
2012Q1   & 26      & 26      & 17      & 84      & 26        & 24        & 34        & 0         & 4       & 412     & (7.00, 2.23)    & 182     & 478     & (1.92, 9.85)    & 256     \\
2012Q2   & 33      & 33      & 23      & 110     & 33        & 31        & 46        & 0         & 4       & 684     & (8.76, 2.33)    & 289     & 810     & (1.94, 12.97)   & 428     \\
2012Q3   & 27      & 27      & 23      & 98      & 27        & 25        & 46        & 0         & 4       & 532     & (7.26, 2.63)    & 196     & 680     & (1.93, 13.44)   & 363    \\
2012Q4	  & 16	  & 16	  & 10	   & 50	   & 16	     & 14	        & 20	      & 0	       & 4	      & 160	 & (4.50, 2.13) 	& 72	    & 188	 & (1.88, 6.50)	 & 104

\end{tabular}}
\end{table}

\begin{table}[!h]
\centering
\caption{Literature instances description.}
\label{table_size_literature_inst}
\resizebox{\textwidth}{!}{
\begin{tabular}{cccccccccccccccc}
Inst & $|I|$ & $|L|$ & $|J|$ & $|A|$ & $|AIL|$ & $|ALL|$ & $|ALJ|$ & $|AIJ|$ & $|K|$ & $|x^s|$ & Avg. Size $x^s$ & $|q^s|$ & $|x^t|$ & Avg. Size $x^t$ & $|q^t|$ \\
\hline
L1   & 3       & 2       & 3       & 9       & 3         & 1         & 3         & 2         & 1       & 10      & (2.50, 2.00)    & 5       & 10      & (2.00, 2.50)    & 5       \\
L2   & 5       & 2       & 4       & 15      & 5         & 2         & 8         & 0         & 4       & 50      & (5.00, 5.00)    & 10      & 28      & (3.50, 4.00)    & 8       \\
L3   & 5       & 2       & 4       & 15      & 5         & 2         & 8         & 0         & 6       & 50      & (5.00, 5.00)    & 10      & 28      & (3.50, 4.00)    & 8       \\
L4   & 8       & 3       & 4       & 26      & 8         & 6         & 12        & 0         & 6       & 144     & (8.00, 6.00)    & 24      & 56      & (4.67, 4.00)    & 12      \\
L5   & 8       & 2       & 5       & 20      & 8         & 2         & 10        & 0         & 4       & 96      & (8.00, 6.00)    & 16      & 50      & (5.00, 5.00)    & 10      \\
L6   & 4       & 2       & 2       & 10      & 4         & 2         & 4         & 0         & 1       & 24      & (4.00, 3.00)    & 8       & 12      & (3.00, 2.00)    & 4       \\
L12  & 3       & 2       & 2       & 9       & 3         & 2         & 4         & 0         & 1       & 18      & (3.00, 3.00)    & 6       & 10      & (2.50, 2.00)    & 4       \\
L13  & 3       & 2       & 2       & 9       & 3         & 2         & 4         & 0         & 1       & 18      & (3.00, 3.00)    & 6       & 10      & (2.50, 2.00)    & 4       \\
L14  & 3       & 2       & 2       & 9       & 3         & 2         & 4         & 0         & 1       & 18      & (3.00, 3.00)    & 6       & 10      & (2.50, 2.00)    & 4       \\
L15  & 3       & 2       & 3       & 18      & 6         & 2         & 6         & 4         & 8       & 24      & (3.00, 4.00)    & 6       & 24      & (4.00, 3.00)    & 6       \\
C2   & 8       & 6       & 6       & 71      & 29        & 9         & 20        & 13        & 4       & 194     & (6.83, 4.83)    & 41      & 203     & (6.33, 5.50)    & 33      \\
D1   & 12      & 10      & 8       & 114     & 46        & 21        & 32        & 15        & 5       & 467     & (9.10, 5.30)    & 91      & 459     & (6.70, 7.00)    & 70     
\end{tabular}}
\end{table}

\begin{table}[!h]
\centering
\caption{Random instances description.}
\label{table_size_random_inst}
\resizebox{\textwidth}{!}{
\begin{tabular}{cccccccccccccccc}
Inst & $|I|$ & $|L|$ & $|J|$ & $|A|$ & $|AIL|$ & $|ALL|$ & $|ALJ|$ & $|AIJ|$ & $|K|$ & $|x^s|$ & Avg. Size $x^s$ & $|q^s|$ & $|x^t|$ & Avg. size $x^t$ & $|q^t|$ \\
\hline
F1   & 10      & 10      & 10      & 84      & 40        & 12        & 32        & 0         & 5       & 241     & (5.60, 4.40)    & 56      & 293     & (5.20, 6.20)    & 62      \\
F2   & 15      & 15      & 15      & 191     & 75        & 35        & 81        & 0         & 3       & 1030    & (10.33, 7.73)   & 155     & 996     & (7.33, 9.87)    & 148     \\
F3   & 15      & 15      & 15      & 177     & 72        & 34        & 71        & 0         & 5       & 1150    & (11.40, 7.00)   & 171     & 1268    & (7.07, 12.27)   & 184     \\
F4   & 15      & 15      & 15      & 187     & 83        & 31        & 73        & 0         & 10      & 1319    & (13.07, 6.93)   & 196     & 1170    & (7.60, 10.67)   & 160     \\
F5   & 20      & 15      & 15      & 183     & 91        & 28        & 64        & 0         & 5       & 1104    & (13.27, 6.13)   & 199     & 999     & (7.93, 8.73)    & 131     \\
F6   & 25      & 15      & 20      & 181     & 71        & 29        & 81        & 0         & 3       & 1613    & (13.73, 7.33)   & 206     & 1540    & (6.67, 16.73)   & 251     \\
F7   & 20      & 15      & 25      & 219     & 82        & 34        & 103       & 0         & 8       & 1926    & (14.53, 9.13)   & 218     & 1934    & (7.73, 18.20)   & 273     \\
F8   & 25      & 15      & 25      & 206     & 90        & 31        & 85        & 0         & 8       & 1617    & (14.73, 7.73)   & 221     & 1543    & (8.07, 13.33)   & 200     \\
F9   & 30      & 10      & 25      & 179     & 94        & 12        & 73        & 0         & 10      & 1292    & (16.30, 8.50)   & 163     & 1181    & (10.60, 12.20)  & 122     \\
F10  & 25      & 15      & 30      & 360     & 72        & 19        & 111       & 158       & 8       & 1155    & (10.40, 8.67)   & 156     & 1013    & (6.07, 12.73)   & 191     \\
F11  & 10      & 10      & 10      & 103     & 32        & 15        & 26        & 30        & 5       & 230     & (6.20, 4.10)    & 62      & 220     & (4.70, 5.20)    & 52      \\
F12  & 15      & 15      & 15      & 194     & 60        & 26        & 55        & 53        & 3       & 697     & (8.80, 5.40)    & 132     & 839     & (5.73, 10.60)   & 159     \\
F13  & 15      & 15      & 15      & 247     & 71        & 32        & 75        & 69        & 5       & 1196    & (11.80, 7.13)   & 177     & 1048    & (6.87, 10.40)   & 156     \\
F14  & 15      & 15      & 15      & 220     & 61        & 29        & 69        & 61        & 10      & 841     & (9.47, 6.53)    & 142     & 835     & (6.00, 10.33)   & 155     \\
F15  & 20      & 15      & 15      & 244     & 80        & 33        & 59        & 72        & 5       & 1147    & (13.47, 6.13)   & 202     & 1066    & (7.53, 10.13)   & 152     \\
F16  & 25      & 15      & 20      & 188     & 59        & 9         & 45        & 75        & 3       & 445     & (7.60, 3.60)    & 114     & 402     & (4.53, 5.33)    & 80      \\
F17  & 20      & 15      & 25      & 167     & 40        & 15        & 39        & 73        & 8       & 313     & (6.20, 3.60)    & 93      & 343     & (3.67, 6.27)    & 94      \\
F18  & 25      & 15      & 25      & 201     & 47        & 17        & 61        & 76        & 8       & 609     & (8.27, 5.20)    & 124     & 557     & (4.27, 9.27)    & 139     \\
F19  & 30      & 10      & 25      & 186     & 41        & 3         & 33        & 109       & 10      & 223     & (6.10, 3.60)    & 61      & 212     & (4.40, 4.50)    & 45      \\
F20  & 30      & 20      & 25      & 220     & 58        & 24        & 59        & 79        & 5       & 509     & (7.65, 4.15)    & 153     & 490     & (4.10, 6.85)    & 137     \\
F21  & 30      & 35      & 25      & 566     & 241       & 117       & 208       & 0         & 5       & 7618    & (24.31, 9.29)   & 851     & 6787    & (10.23, 19.77)  & 692     \\
F22  & 35      & 15      & 40      & 347     & 66        & 12        & 85        & 184       & 8       & 786     & (8.33, 6.47)    & 125     & 781     & (5.20, 9.53)    & 143     \\
F23  & 35      & 40      & 35      & 536     & 215       & 119       & 202       & 0         & 5       & 6634    & (21.93, 8.03)   & 877     & 6274    & (8.35, 20.93)   & 837     \\
F24  & 25      & 15      & 30      & 355     & 78        & 26        & 78        & 173       & 8       & 1338    & (13.87, 6.93)   & 208     & 1330    & (6.93, 13.93)   & 209    
\end{tabular}}
\end{table}

\section*{Acknowledgments}

This work was supported by NSF CMMI [grant number {1562578}]; and the CNPq-Brazil [grant number 248941/2013-5].

\bibliographystyle{plain}
\bibliography{Rankone-bib}

\end{document}